\documentclass{amsart}
\usepackage[utf8]{inputenc}
\usepackage{amsmath}
\DeclareMathOperator\acosh{arccosh}
\DeclareMathOperator*{\argmax}{arg\,max}
\usepackage{amsfonts}
\usepackage{amssymb}
\usepackage{amsthm}
\usepackage{graphicx}
\usepackage{epstopdf}
\usepackage{enumerate}
\usepackage{float}
\usepackage{url}
\usepackage{setspace}
\usepackage{setspace}
\usepackage{enumitem}
\usepackage[dvipsnames]{xcolor}
\usepackage[top=1.5cm,bottom=2cm,left=2cm,right=1.5cm]{geometry}
\usepackage[normalem]{ulem}
\usepackage{cancel}

\makeatletter
\@namedef{subjclassname@2020}{\textup{2020} Mathematics Subject Classification}
\makeatother

\begin{document}
\title[Coupled surface diffusion and mean curvature motion]{Coupled surface diffusion and mean curvature motion: axisymmetric steady states with two grains and a hole}

\author[K. Golubkov]{Katrine Golubkov}
\address{Department of Mathematics, Technion-IIT, Haifa 32000, Israel}
\email{k1a9t9i6a@campus.technion.ac.il}
\author[A. Novick-Cohen]{Amy Novick-Cohen}
\address{Department of Mathematics, Technion-IIT, Haifa 32000, Israel}
\email{amync@technion.ac.il}
\author[Y. Vaknin]{Yotam Vaknin}
\address{Racah Institute of Physics, The Hebrew University of Jerusalem, Jerusalem 91904, Israel}
\email{yotam.vaknin@mail.huji.ac.il}

\thanks{The authors wish to acknowledge the support of the Israel Science Foundation (Grant \# 1200/16).
}
\date{\today}
\begin{abstract}
 The evolution of  two  grains, which lie on a substrate  and are in contact with each other, can be roughly described by a model
 in which the exterior surfaces
of the  grains  evolve by surface diffusion and the  grain boundary, namely the contact surface between  the adjacent  grains, evolves by
motion by mean curvature. We consider an axi-symmetric two grain system, contained within an inert bounding semi-infinite cylinder with a hole along the axis of symmetry. Boundary conditions are imposed
reflecting the considerations of W.W. Mullins, 1958.

We  focus here on the  steady states of this system. At steady state,  the  exterior surfaces  have constant and equal mean curvatures, and
the grain boundary  has zero mean curvature; the exterior surfaces
 are nodoids and the grain boundary surface is  a catenoid. The physical parameters in the model can be expressed via two angles $\beta$ and $\theta_c$, which depend on the surface free energies.
Typically if a steady state solution exists for given values of $(\beta, \theta_c)$, then there exists a continuum
of such solutions. In particular, we prove that there exists a continuum of solutions with $\theta_c=\pi$ for any $\beta \in (\pi/2, \pi)$. 
\end{abstract}

\maketitle

\newcommand{\norm}[1]{\left\lVert#1\right\rVert}
\newcommand\dd{\, \mathrm{d}}
\newcommand\ee{\mathrm{e}}
\newcommand\ess{\mathrm{ess}}
\newcommand\en{\textbf{n}}
\newcommand\vu{\textbf{u}}
\newcommand\vU{\textbf{U}}
\newcommand\vv{\textbf{v}}
\newcommand\Rr{\mathbb{R}}
\newcommand\Nn{\mathbb{N}}
\newcommand{\di}{\mathrm{div}_x}
\newcommand{\intQT}{\int\limits_0^t \!\!\!\int_\Omega}
\newcommand{\intQtau}{\int\limits_0^\hat{\tau} \!\!\!\int_\Omega}
\newcommand{\sgn}{\text{sgn}}
\renewcommand{\theenumi}{\roman{enumi}}   \renewcommand{\labelenumi}{(\theenumi)}

\newtheorem{thm}{Theorem}[section]
\newtheorem{lemma}[thm]{Lemma}
\newtheorem{prop}[thm]{Proposition}
\newtheorem{cor}[thm]{Corollary}
\newtheorem{note}[thm]{Remark}
\newtheorem{claim}[thm]{Claim}

\theoremstyle{definition}
\newtheorem{deff}{Definition}

\newcommand{\Kpurple}[1]{{\color{purple}{#1}}}
\newcommand{\Kgreen}[1]{{\color{green}{#1}}}
\newcommand{\Kblue}[1]{{\color{blue}{#1}}}
\newcommand{\Kmagenta}[1]{{\color{magenta}{#1}}}
\newcommand{\Kred}[1]{{\color{red}{#1}}}

\allowdisplaybreaks

\numberwithin{equation}{section}

\newpage
\section{Introduction}\label{Intro}

The  present paper is inherently interdisciplinary in that we analyse a model designed for the study of certain processes arising in materials science related to the stability of thin solid films;
 while the model and  its analysis are of independent mathematical interest,
   the results from the analysis nevertheless have implications for the original materials science problem.

Thin solid films,  such as  nano-thin metallic and ceramic films, are  commonly used in numerous  applications in electronics, optoelectronics,  sensors, and more recently for anti-bacterial coatings, and
for many applications, stability is a critical concern. Instability issues include wetting, dewetting, and hole formation \cite{Thompson2012,KaplanCWC}; dewetting and hole formation can lead to agglomeration \cite{SrolovitzG}, namely the  break up of regions of the film into
droplets.
 Thin solid films are typically composed of numerous crystals or  grains. Ideally one would like to model systems with  large numbers of grains, but
 to do this faithfully is not an easy task. Moreover, in most realistic physical system, many additional effects are present.
 In terms of modelling, there is a large and growing literature taking into account many effects \cite{Thompson2012,Jiang2019}.
 However, when many effects are included, the analysis frequently relies on numerics and various not overly controlled assumptions and approximations.
 Accordingly, we follow a somewhat different path; we adopt a rather simplistic model in a rather special geometry, but this approach allows us to reach some
 precise and meaningful  conclusions.

 More specifically, we focus  on a system with  two adjacent grains and a hole.
 For simplicity, an axi-symmetric geometry is assumed. The grains are assumed to be supported from below by a substrate and to be contained within an inert cylinder. The region above the grains
 is assumed to be  gaseous or vacuum, and
  the further assumption is made that there is a hole along the axis of symmetry, which exposes the substrate.
 Clearly our two grain model with a hole constitutes a very simplified model for the evolution of a system with a large number of grains and holes,
   but  our conclusions are indicative. For a related numerical model, see e.g. \cite{VEAAJE}.

   Given the geometry described above,     our model for the dynamics of the two grain system can be formulated in terms of a coupled system of partial differential equations, in which  the exterior surfaces of the grains evolve by surface diffusion, and the surfaces of contact between neighboring grains, known as grain boundaries,  evolve by mean curvature motion. Notably,  motion by mean curvature and motion by surface diffusion were both first proposed by the renowned materials scientist, W.W. Mullins \cite{MU6,MU7} in 1956  and 1957, respectively.
The boundary conditions for the system reflect conservation of mass flux, continuity of the chemical potential, and local balance of mechanical forces.

From the  viewpoint of materials science, our model  is quite simple; possible influences from evaporation, bulk diffusion, magnetic response, elastic effects, crystalline anisotropy are neglected.
However, from a mathematical viewpoint, our model is not at all so simple. Both surface diffusion and motion by mean curvature constitute geometric motions, namely  surface  motions where the dynamics is defined in terms of geometric (coordinate independent) quantities, such as the mean curvature, the normal velocity, the surface Laplacian.  While many properties of mean curvature  motion  have been explored in depth \cite{Bellettini2013}, surface diffusion motion has so far received less attention \cite{PrussS2016}, and the literature in regard to the coupling of these motions is considerably thinner \cite{NC2000,LydiaNC,D2016}.  Given the  increasing interest  within the mathematical community in geometric motions in general, which includes, for example,  Ricci and Willmore flows, and given  the implications of these coupled motions in the context of materials science, the literature and overall interest in motions such as these coupled motions clearly stands to grow.


In terms of  the mathematical analysis of the coupled system, clearly it is desirable to address standard issues, such as existence, uniqueness, and regularity.  These issues have been addressed for
certain similar related systems, such as surface diffusion in axisymmetric geometries \cite{axiSimonett_1,axiSimonett_2} or for the coupled motion in the plane of curves evolving by surface diffusion \cite{GarckeNC}, but not to our knowledge for our system per se.
 Such an analysis should be possible and not too problematic;  we demonstrate in \S\ref{parametric formulation} that the coupled motion can be formulated in terms of the motion of three parametric curves and four intersection points
 in a meridian cross-section,
  and within this framework it should be possible  to utilize quasi-linear parabolic theory \cite{GarckeNC,PrussS2016} to establish existence and regularity.
  However, we prefer to sidestep this effort for the moment;  we choose, rather,  to assume the existence of sufficiently smooth solutions and  to explore the qualitative behavior of solutions, in particular to explore some of
  the qualitative features of the steady states.

For the coupled motions,  certain special types of solutions have been considered; for example in \cite{Kanel_1, Kanel_2}, proofs are given of the existence of certain travelling wave solutions for the coupled motions,
 including possibly  ``non-classical'' travelling wave solutions, which are not describable as a graph of a function. Moreover, formal asymptotic calculations indicate that the coupled motions can be viewed as a ``hydrodynamic limit'' of coupled Allen-Cahn/Cahn-Hilliard equations with degenerate mobility, \cite{NC2000,LydiaNC}. For further discussion of approximate travelling wave solutions, spherically capped hexagonal solutions,
  grain annihilation, and more, see  \cite{AnnaZ,D2016,DMNV2017}.  For mean curvature motion and surface diffusion, there has been some emphasis on qualitative geometric properties of solutions;  for example,  do initially embedded solutions   remain embedded?  under what conditions is asymptotic convexity guaranteed? See e.g. \cite{Blatt} and references therein; Similar issues would be interesting to explore within the framework of the coupled motions, however this is beyond the scope of the present paper.

 In the present paper, the geometry of our model and some of the physical background for the problem is outlined in \S\ref{Backg}; in particular axi-symmetry is assumed throughout.   In \S\ref{parametric formulation},  a parametric formulation for the problem in terms of evolving curves and
  intersection points in the meridian cross-section is derived. Assuming  existence and sufficient regularity, conservation of mass  (volume) and energy dissipation for the dynamic problem is demonstrated.
  Afterwards,  we  focus on the set of steady states of the system, and in \S\ref{stationary}, we reduce the characterization of the possible steady states to a parametric description of how the surfaces of constant mean curvature can be combined in accordance with the
  various constraints and boundary conditions. More specifically, within the context of our axi-symmetric setting, this corresponds  to gluing  together two nodoids and a catenoid at designated angles,
  and constraining these surfaces to intersect the substrate or the inert bounding cylinder at prescribed angles. Somewhat similar concerns arise in the context of some capillary problems, \cite{Finn,EHLM}.
  The results here build on the results of Vadim Derkach \cite{D2016}, and  extend an earlier study of a two grain system without a hole, where the characterization of the steady states could be reduced to a parametric coupling of spheres, nodoids, and catenoids \cite{MND2018}.
   We show in Lemma \ref{angle_order} and Lemma \ref{zerothetac}  that  when $\beta=\pi/2$ or $\theta_c=0,$  there do not exist steady state solutions.
   Necessary and sufficient parametric conditions are prescribed for the existence of steady states for $(\beta,\theta_c)\in(\pi/2,\pi)\times(0,\pi],$ and a few solution profiles which have been numerically calculated are portrayed.
   In \S\ref{asymptotic analysis}, we prove that there exists a continuum of
steady solutions with $\theta_c=\pi$ for any $\beta \in (\pi/2, \pi)$, where $\theta_c$ is the contact angle of the inner grain with the substrate and $\beta$ is the angle extended between each of the exterior grain
surfaces and the grain boundary. Our numerical results indicate that similar properties should hold for a wider range of angles.

\section{Background}\label{Backg}
Following the discussion in \S\ref{Intro}, we focus on an axi-symmetric system containing two grains and a hole which lie
within an inert bounding orthogonal  cylinder on a planar substrate.  We assume  that the axis of symmetry  coincides
with the $z$-axis, and that the planar substrate which bounds the system from below is  located at $z=0$. The  region above the grains and the hole is assumed to be occupied by a gas or vacuum.
We shall refer to the grain which lies closer to the $z$-axis as {grain$_1$}, or the \textbf{inner grain}, and to the grain which lies closer to the bounding cylinder as {grain$_2$}, or the \textbf{outer grain}.

Thus, the system contains two grain volumes, which we shall refer to as $V_1$ and $V_2$, respectively, and accordingly we can define the total grain volume as
\begin{equation} \label{total_volume}
\mathbb{V}=V_1 + V_2;
\end{equation}
in \S\ref{parametric formulation},  $\mathbb{V}$ will be shown to be time independent.

In the geometry under consideration, we may identify 7 surfaces. These include the  \textbf{exterior surfaces} of the inner grain and the outer grain, which we shall refer to as $S_1$ and $S_2$, respectively, which separate these grains from the gas or vacuum.  We shall assume that the two grains are in contact
along a \textbf{grain boundary} surface, which we refer to as $S_3$. We view the substrate surface to be partitioned into the region below the hole, which is exposed to the gas or vacuum above, and into the region
which  is covered by the grains; these surface regions are referred to as $S_{4}$ and $S_{5}$, respectively.  Similarly, we can envision the surface of the inert bounding cylinder which lies
above $z=0$ to be partitioned into a region, $S_{6},$  which is in contact with the outer grain and which extends down to the substrate, and into the complement of this surface, $S_{7}$.
The total free energy of the system is given by the  sum of the areas of the various surfaces, weighted by their respective  surface free energies.
Accordingly we let  $\gamma_{ex}$ denote  the surface free energy (free energy/area) of the exterior surfaces, and we let $\gamma_{gb}$ denote the surface free energy of the grain boundary.
Similar we let $\gamma_{gas}$ denote the surface free energy of $S_{4}$, and we let $\gamma_{grs}$ denote the surface free energy of $S_{5}$.
Since we are assuming that the bounding cylinder is inert, we shall assume that the surface free energy of the cylinder vanishes for both $S_{6}$ and its complement $S_{7}$.
  Thus we obtain the following expression for the total free energy of the system,
\begin{equation} \label{totalfE}
\mathbb{E}= \gamma_{ex} (A_{1} + A_{2})  + \gamma_{gb} A_{3} + \gamma_{gas} A_{4} + \gamma_{grs} A_{5},
\end{equation}
where   $A_{i}$ denotes the area of surface $S_i$, for $i=1,...,7;$  note that $\mathbb{E}$ takes into account only surface free energy contributions,
 as other possible energetic contributions have been neglected.
 The energy $\mathbb{E}$ may be time dependent, due to the possible time dependence of the areas $A_i$.
  Throughout, the surfaces $S_i$, $i=1,...,7,$ will be assumed  to be  sufficiently regular hypersurfaces, in accordance with the discussion in
 the Introduction.

\bigskip
Moreover, in accordance with the discussion in the Introduction, the  exterior surfaces, $S_1$, $S_2$ are taken  to evolve by \textbf{motion by surface diffusion} \cite{MU7},
 and the grain boundary surface, $S_3$ is taken to evolve by \textbf{motion by mean curvature} \cite{MU6}. To define these motions,
 for $i \in \{1,2,3\},$ we let
$V_{n_i}$ denote the normal velocity of surface $S_i$, $V_{n_i} := \vec{V}_i \cdot \widehat{N}_i$,   where $\vec{V}_i$  is the surface velocity of $S_i$ and $\widehat{N}_i$ is the outer unit normal to $S_i$.
 Here the outer normals $\widehat{N}_i$ are defined in accordance with the orientation implied by the parametrizations of the surfaces $S_i$, to be introduced shortly.
 Similarly, for $i\in \{1,2,3\},$ we let $H_i$ denote the mean curvature of surface $S_i$, where $H_i$ refers to the average of the principle curvatures.
 In terms of these definitions, motion by mean curvature of the grain boundary may be expressed as
 \begin{equation}\label{meancur}
V_{n_3}=\mathcal{A}H_3,
\end{equation}
 where the coefficient $\mathcal{A}>0$ is  known in the materials science literature as the reduced mobility, whose  dimensions are $[\mathcal{A}]=\frac{L^2}{T}$.
 To formulate motion by surface diffusion of the exterior surfaces, $S_1,$ $S_2$, we let
 $\Delta_s$ denote the Laplace-Beltrami operator, and since $S_1$, $S_2$ have been assumed to be smooth hypersurfaces, $\Delta_s$ may be expressed as \cite{DMNV2017}
\begin{equation}
\Delta_s:=\sum_{j=1}^{2}\frac{\partial^2}{\partial{s_j}^2},
\end{equation}
where $s_j,$ $j=1,2$ are locally defined arc length parametrizations in the direction of the principal curvatures \cite{BO}.
Accordingly, we may formulate the motion by surface diffusion of the exterior surfaces as
\begin{equation}\label{surdiff}
V_{n_i}=-\mathcal{B}\Delta_s H_i,\quad i=1,2,
\end{equation}
where $\mathcal{B}>0$ denotes a kinetic coefficient, known in the materials science literature as the Mullins'  coefficient, whose dimensions are $[\mathcal{B}]=\frac{L^4}{T}$.

\bigskip
Before prescribing boundary and initial conditions, let us describe the geometry of the system which we wish to study  in a bit more detail. Our system is contained in the  semi-infinite  cylinder $\Upsilon$,
\begin{equation} \label{cylinder}
\Upsilon = \{ \;(r,z,\psi)\; | \; 0 \le r \le R, \; 0 \le z, \; 0 \le \psi < 2\pi\; \};
\end{equation}
we  adopt the normalization
\begin{equation}
R=1.\label{R1}
\end{equation}
Since the system under consideration is axi-symmetric, the surfaces
$S_i$, $i=1,...,7$ which were mentioned above and which correspond to hypersurfaces in $\Upsilon$, are $\psi$-independent, and thus determined by the curves
 $\Gamma_i$, $i=1,...,7$ which correspond to their respective projections in the meridian cross-section,
\begin{equation} \label{meridian}
\mathcal{M} =   \{ \;(r,z)\; | \; 0 \le r \le 1, \; 0 \le z \;\}.
\end{equation}

Given our geometric framework, we can distinguish four points of intersection of these curves,
\begin{enumerate} \roman{enumiv}
\item $(r_1^\ast, 0),$  the point of intersection of $\Gamma_1 \cap \Gamma_{4} \cap \Gamma_{5}$, where $\Gamma_1$ meets the substrate,
\item $(1, z_2^\ast),$ the point of intersection of $\Gamma_2 \cap \Gamma_{6} \cap \Gamma_{7}$, where $\Gamma_2$ meets the inert bounding cylinder,
\item $(r_3^\ast, 0),$  the point of intersection of $\Gamma_3 \cap \Gamma_{5}$, where $\Gamma_3$ meets the substrate,
\item $(\bar{r}, \bar{z}),$ the point of intersection of $\Gamma_1 \cap \Gamma_2 \cap \Gamma_3$.
\end{enumerate}
In counting these points, we are inherently assuming that initially and throughout the evolution of the system,
these points of intersection exist, and constitute {the only points of intersection}; in particular that there are no self-intersections and that the hole remains exposed. These assumptions can be understood to be reasonable for some, but not all, initial
curves \cite{Blatt}. This implies in particular the assumption that the outer grain
is actually in contact with the inert boundary of the cylinder.
These assumptions imply in particular that
\begin{equation} \label{constraints_intersections}
0 < r_1^\ast < r_3^\ast<1, \quad 0<z_2^\ast, \quad 0<\bar{r} <1, \quad 0<\bar{z}.
\end{equation}
Under these assumptions, it follows that since the substrate and the boundary of the cylinder are fixed and do not move, the two grain system configuration
can be specified in terms of the three curves $\Gamma_i,$ $i=1,2,3,$ and the four intersection points, listed above, which may all be possibly time dependent.
Given that the four points of intersection  uniquely determine the end points of the curves $\Gamma_i$, $i=1,2,3,$ (assumed to be continuous and rectifiable),
at any given time we may parametrize each of the three curves in terms of  $\alpha \in [0,1],$ with $\alpha=0$ corresponding to
$(r_1^\ast, 0)$ for
$\Gamma_1,$ to  $(\bar{r}, \bar{z})$ for $\Gamma_2,$ and to $(r_3^\ast, 0)$ for
$\Gamma_3$, and with $\alpha=1$ corresponding to $(\bar{r}, \bar{z})$ for  $\Gamma_1$ and for $\Gamma_3$, and to $(1, z_2^\ast)$ for  $\Gamma_2.$
 Since this paper is primarily devoted to the study of the steady states, we refrain for the most part from indicating time dependence explicitly.
The unique point $(\bar{r}, \bar{z}),$ where the three curves $\Gamma_i,$ $i=1,2,3,$ intersect, is often referred to as the
\textbf{triple junction}. Note  that  the ``points'' of intersection in $\mathcal{M}$,  correspond to ``circles'' of intersection on the surfaces $S_i$ in $\Upsilon$.

\bigskip

In discussing the configuration and in particular the curves $\Gamma_i$, $i=1,2,3,$ it shall be at times convenient later  to introduce  angle descriptions
rather than to use the $\alpha$ parametrizations. However, with regard to the $\alpha$  parametrizations of the curves, we may define them more explicitly by requiring that
\begin{equation} \label{equidistribution}
((r_{i\alpha})^2 + (z_{i\alpha})^2)_\alpha =0, \quad \alpha \in [0,1], \quad\quad i=1,2,3,
\end{equation}
where $r_{i\alpha}=\frac{\partial}{\partial \alpha}r_i,$  $z_{i\alpha}=\frac{\partial}{\partial \alpha}z_i$. This implies that the $\alpha$ parametrizations \color{black} correspond to normalized arc-length parametrizations, which were seen to be helpful in numerical implementations  \cite{PanWetton,D2016,DerkachJEPE,DerkachScripta}. Given the parametric description above, and given knowledge of at least one of the endpoints on each curve, for $i=1,2,3$
the curve $\Gamma_i$    may  be uniquely prescribed in terms of the angles $\theta_i(\alpha),$ $\alpha\in [0,1],$ where $\theta_i(\alpha)$ is the angle between
$\hat{\tau}_i(\alpha),$ the unit tangent to $\Gamma_i$ at $\alpha$ in the direction of the parametrization, and the (positive) $r$-axis. In particular, this approach allows us to identify
the angles $\theta_i(0)$, $\theta_i(1)$, $i=1,2,3,$ at the various points of intersection. We shall adopt the notation
\begin{equation} \label{tjangles}
\bar{\theta}_1= \theta_1(1), \quad \bar{\theta}_2=\theta_2(0), \quad \bar{\theta}_3=\theta_3(1),\end{equation}
at the triple junction, in analogy with the notation $(\bar{r}, \bar{z})$ introduced above.

 \subsection{Boundary conditions}

Following Mullins 1958 \cite{MU8,Kanel_0}, we identify four types of boundary conditions: persistence, balance of mechanical forces, continuity of the chemical potential, and balance of mass flux,
which are described in detail below.

\smallskip\par\noindent\textbf{Persistence:} Persistence  here refers to the assumption that the configuration described  above is maintained throughout the evolution; namely that the system remains axi-symmetric and can be described
in the meridian cross-section via the three curves $\Gamma_i$, $i=1,2,3,$ and the four intersection points which were enumerated above. From a physical viewpoint, this assumption is reasonable, since
if one of the curves were, for example, to detach from one of the intersection points, the physical interpretation of the configuration would be lost. Our assumption of regularity outlined earlier  implies that (\ref{constraints_intersections}) should hold continuously; if one of the conditions ceases to hold, this can be viewed
as a break-down of the requirements and indicates termination  of the evolution under consideration. In particular, we assume that  the hole, the exterior surfaces and the grain boundary surface  persist, and that no other holes or grains  form.

\smallskip\par\noindent\textbf{Balance of mechanical forces:} This condition is often referred to as \textbf{Herring's law} or \textbf{Young's law}, depending on whether solids or fluids are being discussed.
It reflects the assumption that the  forces at the intersection
points are determined by the local atomic forces in its direct vicinity, and typically allows local force equilibrium to be attained before the surrounding surfaces have reached equilibrium.
Sometimes this condition is relaxed, but it
constitutes a simple and common assumption. This condition in a general setting may be stated as follows: given  $n\in \mathbb{N}$ surfaces that meet freely along a given curve of intersection,
\begin{equation} \label{Herring}
\sum_{j=1}^n \gamma_j \hat{\tau}_j =0,
\end{equation}
where for $j=1,2,\ldots,n$, $\gamma_j$  denotes  the surface free energy of the $j^{th}$ surface, and $\hat{\tau}_j$ denotes the unit tangent to the $j^{th}$ surface which is orthogonal to the curve of intersection and points
outwards from the intersection point. In the present context, this condition applies at the triple junction. At the constrained intersection points $(r_1^\ast, 0),(1, z_2^\ast),(r_3^\ast, 0),$ which lie on either the inert cylinder or the substrate which are taken to be fixed, balance of the tangential mechanical forces can be assumed, namely
\begin{equation} \label{HerringConstrained}
\hat{t}\cdot\left(\sum_{j=1}^n \gamma_j \hat{\tau}_j \right)=0,
\end{equation}
where $\hat{t}$ is a unit tangent vector to the fixed surface at the point of contact in the meridian cross section.
 \par Thus, taking into account the geometry and the surface free energies \cite{KaplanCWC} and given that $\gamma_1=\gamma_2=\gamma_{ex}$, $\gamma_3=\gamma_{gb},$  (\ref{Herring}) implies that
 \begin{equation}\label{beta1}
 -\bar{\theta}_1+\bar{\theta}_3  = \beta,
 \end{equation}
 \begin{equation}\label{beta2}
 \bar{\theta}_2-\bar{\theta}_3 + \pi = \beta,
 \end{equation}
 where $\frac{\pi}{2} \le \beta:=\arccos({-\gamma_{gb}}/({2\gamma_{ex}})) \le \pi$,  $0\le \gamma_{gb}/\gamma_{ex} \le 2.$
 The  case in which $\beta=\pi$ and $ \gamma_{gb}= 2\gamma_{ex}$ corresponds to the singular case of complete wetting at the triple junction \cite{KaplanCWC,LydiaNC},  and we shall limit our considerations  to the range
 \begin{equation}
 \frac{\pi}{2}\le \beta <\pi, \quad \beta:=\arccos({-\gamma_{gb}}/({2\gamma_{ex}})),\quad 0\le \gamma_{gb} < 2\gamma_{ex}. \label{gammagamma}
 \end{equation}
 Note that (\ref{gammagamma}) implies  that
 \begin{equation}\label{2xy}
 \cos(\beta)\le 0, \quad
 \sin(\beta)>0.
 \end{equation}
 At the constrained intersection points, (\ref{HerringConstrained}) implies that
\begin{equation}\label{inex0}
\theta_1(0)=\theta_c:=\arccos((\gamma_{gas}-\gamma_{grs})/{\gamma_{ex}}),
\end{equation}
where, since the parametric curve $\Gamma_1=\Gamma_1(\alpha)$ lies in $\mathcal{M}$,
\begin{equation}
0 \le \theta_c \le \pi,    \label{thetac}
\end{equation}
inertness of the bounding cylinder and (\ref{HerringConstrained}) imply that
\begin{equation}
\theta_2(1)=0, \label{outex1}
\end{equation}
and planarity of the substrate and (\ref{HerringConstrained}) imply that
\begin{equation}
\theta_3(0)={\pi}/{2}. \label{gb0}
\end{equation}

\smallskip\par\noindent\textbf{Continuity of the chemical potential:} Following \cite{MU7}, continuity of the chemical potential  is assumed, and within the framework of our problem,  the
chemical potential is proportional to the mean curvature along the exterior surfaces. In the spirit of the notation introduced earlier, this implies that
\begin{equation} H_1(1)=H_2(0), \label{cont_chem_pot}
\end{equation}
 where for $i=1,2,$  $H_i(\alpha)$ refers to the mean curvature of $S_i$ evaluated along  $\Gamma_i$  at $\alpha$.

\smallskip\par\noindent\textbf{Balance of the mass flux along the exterior surfaces:}
 In formulating the balance of mass flux along the exterior surfaces, we  neglect possible mass flux contributions from along the grain boundary onto exterior surfaces at the triple junction; the accuracy of this assumption depends on the physical system under consideration.
 Following \cite{MU8,Kanel_0}, the mass flux along the exterior surface can be
identified as
\begin{equation} \label{mass_flux}
J = - \nabla_s \mu, \quad \mu = H, \end{equation}
where $\nabla_s$ is the surface gradient.  The parametric formulation of this boundary condition is given in \S\ref{parametric formulation}.

\subsection{Initial conditions}\label{initiall conditions}
It seems reasonable to prescribe initial conditions  which would allow our  configurational assumptions to be maintained for all time, permitting steady state to be asymptotically attained.
Thus it would seem sensible to choose initial conditions within the geometric framework outlined above,  which satisfy the constraints and  boundary conditions.
In particular, it would seem  reasonable to consider, for example, appropriately defined axi-symmetric perturbations of some steady states.
An alternative approach could be to be guided by the notion of a critical radius $r_c$ \cite{SSafran2,REFF}, according to which holes whose radius is smaller than $r_c$ tend to close, and holes whose radius is larger than $r_c$ tend to grow, where $r_c$ can be roughly determined via energetic considerations, given a film with average height $h$, \cite{WangJiangBaoS,DMNV2017}.
Within this context,  assuming that $0 < r_c  \ll 1=R,$ we might define an initial configuration in which $r_1^\ast =2r_c$ and  $r_3^\ast=2 r_c +h$, where $h$ roughly corresponds to the average height of both grains, taking care to satisfy the boundary conditions for some set of angles $\theta_c$, $\beta,$  chosen in accordance with  (\ref{thetac}),  (\ref{gammagamma}).

\vspace{5mm}
\section{The Dynamic Problem Formulation in Terms of Parameteric Curves}\label{parametric formulation}

Following the  discussion earlier, we assume that  the solutions are axi-symmetric throughout the evolution, and that they belong to the class of configurations described in \S\ref{Backg}.
This implies in particular that they can be described by their profile
 in the meridian cross-section,
\begin{equation} \label{meridian1}
\mathcal{M} =   \{\; (r,z)\; | \; 0 \le r \le 1, \; 0 \le z \; \},
\end{equation}
where the normalization (\ref{R1})  has been adopted,
via three curves and four points of intersection. Since the problem here  is  dynamic,
we  consider the  three curves $\Gamma_i,$ $i=1,2,3,$ to be time dependent,
where $\Gamma_1(t)$ and $\Gamma_2(t)$ refer, respectively, to the inner and outer exterior surfaces, and $\Gamma_3(t)$ refers to the grain boundary surface,
and the four evolving points which were defined in \S\ref{Backg}, namely
\begin{equation} \label{points}
 \hbox{i})\quad (r_1^\ast(t), 0),\quad\quad  \hbox{ii})\quad (1, z_2^\ast(t)),\quad\quad \hbox{iii})\quad (r_3^\ast(t), 0),\quad\quad \hbox{iv})\quad(\bar{r}(t), \bar{z}(t)),
 \end{equation}
  and we impose the constraints  (\ref{constraints_intersections}).
We assume throughout that the exterior surfaces and the grain boundary curves are sufficiently regular, and moreover that the evolution persists within the class of configurations
considered for  times $t \in [0,T),$ $T>0,$ with some implicit emphasis on the case $T=\infty$.
If our assumptions and constraints  cease to exist at some given positive time, we can take this as an upper limit on the stopping time for the evolutionary problem.
%
%
%
%
%
\subsection{Equations of motion} \label{eqns_motion}
\smallskip
For $t\in [0, T)$ and $i=1,2,3,$ we set
\begin{equation} \label{parG3}
\Gamma_i(t)=\{(r_i(\alpha, t),\, z_i(\alpha,t)) \, | \, \alpha \in [0, 1] \, \},
\end{equation}
where the parametrizations correspond to the normalized arc-length parametrizations (\ref{equidistribution}), which satisfy
\begin{equation}\label{ugs}
(r_{i\alpha} \, r_{i\alpha\alpha} + z_{i\alpha} \, z_{i\alpha\alpha})(\alpha,t) =0, \quad \alpha \in [0,1],\quad t \in [0,T), \quad\quad i=1,2,3.
\end{equation}
 Setting 
 \begin{equation} \label{tau}
\hat{\tau}_i(\alpha,t)=\left.\Biggl(\frac{r_{i\alpha}}{\sqrt{r_{i\alpha}^2+z_{i\alpha}^2}},\frac{z_{i\alpha}}{\sqrt{r_{i\alpha}^2+z_{i\alpha}^2}}\Biggr)\right|_{(\alpha,t)}, \quad\quad \alpha \in [0,1],\quad t \in [0,T), \quad\quad i=1,2,3,\end{equation}
we may prescribe the following angle descriptions of the curves $\Gamma_i(t),$ $i=1,2,3,$ 
\begin{equation} \label{angles}
\cos(\theta_i(\alpha,t))= \hat{\tau}_i(\alpha,t) \cdot \hat{e}_r,\quad \sin(\theta_i(\alpha,t))=\hat{\tau}_i(\alpha,t) \cdot \hat{e}_z,
\quad i=1,2,3, \quad \alpha \in [0,1],\quad t \in [0,T),
\end{equation}
  \color{black} where $\theta_i(0,t) \in (-\pi, \pi]$, $t\in [0,T),$ $i=1,2,3,$ and $\hat{e}_r=(1,0),\ \hat{e}_z=(0,1)$ are the unit vectors in the directions of the (positive) $r$-axis and $z$-axis, respectively.  
Assuming that
$\theta_i(\alpha,t) \in C([0,1] \times [0,T)),$  $i=1,2,3,$ and since    self-intersecting configurations have been ruled out, in looking at the
boundary conditions given below in  (\ref{r1z1})--(\ref{r3z3}), it seems reasonable to assume that $\theta_i(\alpha,t) \in (-\pi, \pi]$, for $\alpha\in (0,1],$ $t\in [0,T),$  and $i=1,2,3.$


From  (\ref{parG3}), we obtain (see e.g. \cite{D2016}) that for $\Gamma_i(t),$ $i=1,2,3$, the normal velocity and the mean curvature may be expressed
in accordance with the description $\Gamma(t)=\{(r(\alpha, t),\, z(\alpha,t)) \, | \, \alpha \in [0, 1] \, \}$  as
\begin{equation}
V_n=\frac{r_{\alpha}z_t-z_{\alpha}r_t}{\sqrt{r_{\alpha}^2+z_{\alpha}^2}},\quad
H=\frac{r_{\alpha}z_{\alpha\alpha}-z_{\alpha}r_{\alpha\alpha}}{2(r_{\alpha}^2+z_{\alpha}^2)^{\frac{3}{2}}}+
\frac{z_\alpha}{2r\sqrt{r_{\alpha}^2+z_{\alpha}^2}},\quad \alpha \in [0,1],\quad t \in [0,T), \label{H3}
\end{equation}
where $V_n=V_n(\alpha,t)$ and $H=H(\alpha,t)$.
From (\ref{surdiff}), (\ref{ugs})--(\ref{H3}), we obtain   for the evolution of the exterior surfaces curves,  $\Gamma_1(t)$ and $\Gamma_2(t)$, that
\begin{equation}
r_{i\alpha}z_{it}-z_{i\alpha}r_{it}=\frac{-\mathcal{B}(r_iH_{i\alpha})_\alpha}{r_i\sqrt{r_{i\alpha}^2+z_{i\alpha}^2}},\quad i=1,2,\quad \alpha\in (0,1),\quad t \in (0,T), \label{va1}
\end{equation}
see \cite[\S1.13.1]{D2016} for details. \\
Similarly, from (\ref{meancur}), and (\ref{ugs})--(\ref{H3}), we obtain  for
the evolution of the grain boundary, $\Gamma_3(t)$, that
\begin{equation}
r_{3\alpha}z_{3t}-z_{3\alpha}r_{3t}=\mathcal{A}\sqrt{r_{3\alpha}^2+z_{3\alpha}^2}H_3,\quad \alpha\in(0,1),\quad t\in (0,T). \label{vb1}
\end{equation}

\subsection{Boundary conditions}\label{boundary_conditions}

We now formulate the boundary conditions discussed in \S\ref{Backg}, in terms of the parametric description given above.

The notion of persistence introduced in \S\ref{Backg} implies the following constraints on the endpoints of the curves,
\begin{eqnarray}
(r_1(0,t),z_1(0,t))&=&(r_1^*(t),0),\quad\;\; (r_1(1,t),z_1(1,t))=(\bar{r}(t),\bar{z}(t)),\quad\ t\in [0,T),\label{r1z1}\\
(r_2(0,t),z_2(0,t))&=&(\bar{r}(t),\bar{z}(t)),\quad (r_2(1,t),z_2(1,t))=(1,z_2^*(t)),\quad\;\;\ t\in [0,T),\label{r2z2}\\
(r_3(0,t),z_3(0,t))&=&(r_3^*(t),0),\quad\;\; (r_3(1,t),z_3(1,t))=(\bar{r}(t),\bar{z}(t)),\quad\ t\in [0,T),\label{r3z3}\\
0<r_1^\ast <r_3^\ast <1,&&\, 0<z_2^\ast,\quad 0<\bar{r}<1,\quad 0<\bar{z},\quad t\in [0,T). \label{r4z4}
\end{eqnarray}

The  balance of mechanical forces implies that
\begin{eqnarray}
&\hat{\tau}_1(0,t) \cdot \hat{e}_r=\cos(\theta_c),\quad\quad\quad\quad\quad\quad\quad   \theta_1(0,t)=\theta_c, \quad\quad\quad\quad  &t \in [0,T),\label{6bc}\\
&\hat{\tau}_2(1,t) \cdot \hat{e}_z=0,\quad\quad\quad\quad\quad\quad   \theta_2(1,t)=0, \quad\quad\quad\quad  &t \in [0,T),\label{7bc}\\
&\hat{\tau}_3(0,t) \cdot \hat{e}_r=0,\quad\quad\quad\quad\quad\quad   \theta_3(0,t)=\frac{\pi}{2}, \quad\quad\quad\quad  &t \in [0,T),\label{8bc}\\
&\hat{\tau}_1(1,t)\cdot \hat{\tau}_3(1,t)=\cos(\beta),\quad\quad  -\theta_1(1,t)+ \theta_3(1,t)=- \bar{\theta}_1(t)+ \bar{\theta}_3(t)=\beta, \quad\quad  &t \in [0,T),\label{9abc}\\
 &-\hat{\tau}_2(0,t)\cdot \hat{\tau}_3(1,t)=\cos(\beta),\quad\quad  \theta_2(0,t)-\theta_3(1,t)=\bar{\theta}_2(t)- \bar{\theta}_3(t)=-\pi+\beta, \quad\quad  &t \in [0,T),\label{9bc}
 \end{eqnarray}
 where $\theta_c$ satisfies (\ref{inex0}), (\ref{thetac}) and $\beta$ satisfies (\ref{gammagamma}). 
 
Continuity of the chemical potential along the exterior surface, (\ref{cont_chem_pot}), implies that
\begin{equation}
H_1(1,t)=H_2(0,t),\ t \in  [0,T). \label{10bc}
\end{equation}

Balance of the mass flux,  (\ref{mass_flux}), along the exterior surface, with no incoming or outgoing mass flux from the grain boundary, or into the substrate or out through the bounding cylinder, implies that  
\begin{equation} \label{bca}
H_{1\alpha}(0,t)=0, \quad
\left.\frac{H_{1\alpha}}{\sqrt{r_{1\alpha}^2 + z_{1\alpha}^2}}\right|_{(1,t)}=\left.\frac{H_{2\alpha}}{\sqrt{r_{2\alpha}^2 + z_{2\alpha}^2}}\right|_{(0,t)},
   \quad H_{2 \alpha}(1,t)=0, \quad t \in [0,T).
\end{equation}

  The assumption that the hole, the exterior surfaces and the grain boundary surface persist and that no additional
  holes or grains form, implies the following constraints:
\begin{eqnarray}
0<r_1(\alpha,t),\,r_3(\alpha,t) <1,\quad &\alpha\in[0,1],\ &t \in [0,T),\label{C11}\\
0<r_2(\alpha,t)< 1,\quad &\alpha\in[0,1),\ &t \in [0,T), \label{C11a}\\
0< z_1(\alpha,t),\,z_3(\alpha,t),\quad &\alpha\in(0,1],\ &t\in [0,T), \label{l4}\\
0<z_2(\alpha,t),\quad &\alpha\in[0,1],\ &t\in [0,T).\label{l4a}
\end{eqnarray}
 Lastly we recall that we do not allow self-intersection of the curves or the appearance of additional intersections.

\subsection{Initial conditions} \label{initial conditions} For $i=1,2,3,$
\begin{equation}\label{ic}
(r_{i}(\alpha,0),z_{i}(\alpha,0))=(r^i_{init}(\alpha),z^i_{init}(\alpha)),\quad \alpha \in [0,1],
\end{equation}
where  $r^i_{init}$, $z^i_{init}$ are prescribed sufficiently smooth  functions which satisfy the constraints and boundary conditions.
\subsection{In summary} \label{problem formulation} The complete dynamic problem formulation is given in (\ref{parG3})--(\ref{ic}) above.

\vspace{5mm}
\subsection{The Total Energy and Volume of the Physical System}\label{volume and energy}

 We now demonstrate two rather intuitive  properties of the evolutionary system,
 namely  that the total energy is non-increasing and that the total volume is conserved.

Recalling (\ref{totalfE}), the total free energy  $\mathbb{E}(t)$ of the system may be expressed as
\begin{equation}\nonumber
\mathbb{E}(t)=\gamma_{gas}A_{gas}(t)+\gamma_{grs}A_{grs}(t)+\gamma_{ex}A_{ex}(t)+\gamma_{gb}A_{gb}(t),\ t\ge 0.
\end{equation}
    Using (\ref{R1}), (\ref{inex0}), in our geometry  $A_{gas}+A_{grs}=\pi R^2=\pi$,   $A_{gas}(t)=\pi r_1^{*2}(t),$ and $\gamma_{gas}-\gamma_{grs}=\gamma_{ex}\cos(\theta_c).$ Hence
\begin{equation}
\mathbb{E}(t)=\pi\gamma_{grs}+\gamma_{ex}(A_{ex}(t)+\cos(\theta_c)\pi r_1^{*2}(t))+\gamma_{gb}A_{gb}(t),\label{etot}
\end{equation}
where 
\begin{equation} \label{aex}
A_{ex}(t)=2\pi\sum_{i=1,2}{\int_{0}^{1}{r_i\sqrt{r_{i\alpha}^2+z_{i\alpha}^2}d\alpha}}, \quad
A_{gb}(t)=2\pi\int_{0}^{1}{r_3\sqrt{r_{3\alpha}^2+z_{3\alpha}^2}d\alpha}.
\end{equation}

\begin{lemma}\label{dedt} {Assuming  sufficient  regularity,
the total free energy,  $\mathbb{E}(t)$, is non-increasing with respect to time.}
\end{lemma}

\textit{Proof.} Assuming sufficient regularity, we get from (\ref{etot}), (\ref{aex})  that
\begin{equation}
\frac{d\mathbb{E}}{dt}=\gamma_{ex}\left(\frac{dA_{ex}}{dt}+\cos(\theta_c)2\pi r_1^*\frac{dr_1^*}{dt}\right)+\gamma_{gb}\frac{dA_{gb}}{dt},\ t\in (0,T).\label{detot}
\end{equation}
 Letting  $A_1,\,  A_2,\, A_3,$ denote the surface areas of the inner exterior surface, of the outer exterior surface and of the grain boundary, respectively,  
 $$\frac{d {A_i}}{dt}=2\pi\int_{0}^{1}{\Big(r_{it}\sqrt{r_{i\alpha}^2+z_{i\alpha}^2}+r_i\Big(\frac{r_{i\alpha}r_{i\alpha t}+z_{i\alpha}z_{i\alpha t}}{\sqrt{r_{i\alpha}^2+z_{i\alpha}^2}}\Big)\Big)d\alpha}, \quad i=1,2,3, \quad t \in (0,T).$$
 Integration by parts implies that
\begin{equation}\label{ai}
2\pi\int_{0}^{1} \Big( r_i
\Big(\frac{r_{i\alpha} r_{i\alpha t}+z_{i\alpha}z_{i\alpha t}}{\sqrt{r_{i\alpha}^2+z_{i\alpha}^2}}\Big)\Big) d\alpha
=- 2\pi\int_{0}^{1} \Big(r_{it}\Big(\frac{r_ir_{i\alpha}}{\sqrt{r_{i\alpha}^2+z_{i\alpha}^2}}\Big)_\alpha+z_{it}\Big(\frac{r_iz_{i\alpha}}{\sqrt{r_{i\alpha}^2
+z_{i\alpha}^2}}\Big)_\alpha\Big)d\alpha +a_i,  \quad i=1,2,3,
\end{equation}
where
\begin{equation}\nonumber
a_i=2\pi\Big(\frac{r_ir_{i\alpha}r_{it}+r_iz_{i\alpha}z_{it}}{\sqrt{r_{i\alpha}^2+z_{i\alpha}^2}}\Big)|_{\alpha=0}^{\alpha=1}.
\end{equation}
It follows from (\ref{gammagamma}), (\ref{9abc})-(\ref{9bc}) that
$\cos\left(\bar{\theta}_
1\right)+ \gamma_{gb}/\gamma_{ex}\cos\left(\bar{\theta}_3\right)=\cos\left(\bar{\theta}_2\right),$
$\sin\left(\bar{\theta}_1\right)+\gamma_{gb}/\gamma_{ex}\sin\left(\bar{\theta}_3\right)=\sin\left(\bar{\theta}_2\right),$
and hence
\begin{equation}\label{ai2}
\sum_{i=1,2}{\gamma_{ex} a_i}+ {\gamma_{gb} a_3}=-2\pi\gamma_{ex}\cos(\theta_c)r_1^*\frac{dr_1^*}{dt}.
\end{equation}
From (\ref{detot})--(\ref{ai2}),
\begin{align}\nonumber
\begin{split}
\begin{array}{ll}
\frac{d\mathbb{E}}{dt}&=2\pi\gamma_{ex}\sum_{i=1,2}\int_{0}^{1}\Big(r_{it}\sqrt{r_{i\alpha}^2+z_{i\alpha}^2}
-r_{it}\Big(\frac{r_ir_{i\alpha}}{\sqrt{r_{i\alpha}^2+z_{i\alpha}^2}}\Big)_\alpha-z_{it}\Big(\frac{r_iz_{i\alpha}}{\sqrt{r_{i\alpha}^2+z_{i\alpha}^2}}\Big)_\alpha\Big)d\alpha\\
\quad &\; + \;
2\pi \gamma_{gb} \int_{0}^{1}\Big(r_{3t}\sqrt{r_{3\alpha}^2+z_{3\alpha}^2}
-r_{3t}\Big(\frac{r_3r_{3\alpha}}{\sqrt{r_{3\alpha}^2+z_{3\alpha}^2}}\Big)_\alpha-z_{3t}\Big(\frac{r_3z_{3\alpha}}{\sqrt{r_{3\alpha}^2+z_{3\alpha}^2}}\Big)_\alpha\Big)d\alpha
.\label{detot2}
\end{array}
\end{split}
\end{align}
Using now (\ref{ugs}) and (\ref{H3})--(\ref{vb1}), as well as (\ref{10bc}) and (\ref{bca}), 
\begin{align}\nonumber
\begin{split}
\begin{array}{ll}
\frac{d\mathbb{E}}{dt}&=4\pi\sum_{i=1,2}\gamma_{ex}\int_{0}^{1}r_iH_i(r_{it}z_{i\alpha}-z_{it}r_{i\alpha})d\alpha  +
4\pi\gamma_{gb}\int_{0}^{1}r_3 H_3(r_{3t}z_{3\alpha}-z_{3t}r_{3\alpha})d\alpha\\[1.5ex]
&=4\pi\Big(\sum_{i=1,2}\gamma_{ex}\int_{0}^{1}r_iH_i\Big(\frac{\mathcal{B}(r_iH_{i\alpha})_\alpha}{r_i\sqrt{r_{i\alpha}^2+z_{i\alpha}^2}}\Big)d\alpha
-\gamma_{gb}\int_{0}^{1}\mathcal{A}r_3H_3^2\sqrt{z_{3\alpha}^2+r_{3\alpha}^2}d\alpha\Big)\\[1.5ex]
&=-4\pi\left(\sum_{i=1,2}\gamma_{ex}\int_{0}^{1}\frac{\mathcal{B}r_iH_{i\alpha}^2}{\sqrt{r_{i\alpha}^2+z_{i\alpha}^2}}d\alpha+\gamma_{gb}\int_{0}^{1}\mathcal{A}r_3H_3^2\sqrt{z_{3\alpha}^2
+r_{3\alpha}^2}d\alpha\right)\le0,\quad t>0.
\end{array}
\end{split}
\end{align}
\hfill$\square$\\\\
\begin{lemma}\label{dvdt} {Assuming sufficient regularity, the total volume $\mathbb{V}(t)$ enclosed by the two grains is time
independent.} 
\end{lemma}
\textit{Proof.}
Following \cite{D2016}, the total volume may be expressed as
\begin{equation}
\mathbb{V}(t)=2\pi\sum_{i=1,2}\int_0^1z_ir_ir_{i\alpha}d\alpha.\label{volume}
\end{equation}
 Differentiating $\mathbb{V}(t)$, then integrating by parts and using (\ref{ugs}), (\ref{va1}), and (\ref{bca}), gives us that for $t>0$,
$$\frac{d \mathbb{V}}{dt}=2\pi\sum_{i=1,2}\int_0^1(z_{it}r_{i\alpha}-z_{i\alpha}r_{it})r_i d\alpha
=2\pi\sum_{i=1,2}\int_0^1 \Bigg(\frac{-\mathcal{B}\left(r_iH_{i\alpha}\right)_\alpha}{\sqrt{r_{i\alpha}^2+z_{i\alpha}^2}}\Bigg)d\alpha=\left.2\pi\sum_{i=1,2}
\frac{-\mathcal{B}r_iH_{i\alpha}}{\sqrt{r_{i\alpha}^2+z_{i\alpha}^2}}\right|_{\alpha=0}^{\alpha=1}=0.$$

\hfill$\square$

\vspace{5mm}
\section{Steady State Solutions}\label{stationary}


\vspace{1.5mm}
\subsection{The Steady State Problem}\label{mc_ss}

The steady state problem formulation, which we summarize below, follows  upon setting $r_{i\, t}(\alpha,t)\equiv z_{i\, t}(\alpha,t)\equiv 0,$ $i=1,2,3,$
(i.e.; $\frac{d}{dt}\Gamma_{i}(t)\equiv 0,$ $i=1,2,3,$) in (\ref{ugs})--(\ref{l4a}).

From   (\ref{C11})--(\ref{l4a}),
\begin{equation} \label{0r1}
0 < r_1(\alpha),\ r_3(\alpha)<1,\quad  \alpha \in [0,1],\qquad 0< r_2(\alpha)< 1,\quad \alpha \in [0,1),
\end{equation}
\begin{equation} \label{0z}
0<z_1(\alpha),\ 0<z_3(\alpha),\quad\alpha\in(0,1],\qquad 0<z_2(\alpha), \quad \alpha\in[0,1],
\end{equation}
and from (\ref{r1z1})--(\ref{r3z3}),
\begin{equation} \label{bc_S1}
\begin{array}{lll}
(r_1,z_1)(0)=(r_1^\ast,0), \quad  &(r_2,z_2)(0)=(\bar{r},\bar{z}), \quad   &(r_3,z_3)(0)=(r_3^\ast,0),\\[1.5ex]
(r_1,z_1)(1)=(\bar{r},\bar{z}), \quad  &(r_2,z_2)(1)=(1,z_2^\ast), \quad   &(r_3,z_3)(1)=(\bar{r},\bar{z}).
\end{array}
\end{equation}
Note that $r_1^\ast,$  $r_3^\ast,$ $\bar{r}$, and $z_2^\ast$, $\bar{z}$ are now time independent, and (\ref{0r1})--(\ref{bc_S1}) together with (\ref{r4z4}) imply  that
\begin{equation} \label{r0z0}
0 < r_1^\ast < r_3^\ast <1, \quad 0<z_2^\ast,\quad  0< \bar{r}<1, \quad 0<   \bar{z},
\end{equation}
\begin{equation} \label{endsne}
(r_i, z_i)(0) \ne (r_i, z_i)(1), \quad i=1,2,3.
\end{equation}
Upon integrating  (\ref{ugs}) with respect to $\alpha$, it  follows from (\ref{endsne})  that
\begin{equation} \label{const_S}
r_{i \alpha}^2 + z_{i \alpha}^2=c_i,\quad \alpha\in [0,1],\quad\quad 0<c_i\in \mathbb{R}, \quad i=1,2,3.
\end{equation}

Recalling that $\mathcal{A},\, \mathcal{B}>0,$ we obtain from  (\ref{const_S}) and  (\ref{va1}), (\ref{vb1}),  that
\begin{equation}\nonumber
(r_iH_{i\alpha})_\alpha=0,\quad i=1,2,\quad H_3=0,\quad\alpha\in (0,1).
\end{equation}
Integrating the first equation above twice with respect to $\alpha,$  we obtain by using (\ref{10bc})--(\ref{bca}) that
\begin{equation}\label{conH}
H_1(\alpha)=H_2(\alpha)=\lambda\in \mathbb{R},\quad H_3(\alpha)=0,\quad \alpha\in[0,1],
\end{equation}
where  $\lambda$ remains to be determined. From (\ref{ugs}), (\ref{H3}),
\begin{eqnarray}
H_i(\alpha)&=&\frac{r_{i \alpha}z_{i \alpha\alpha}-z_{i \alpha}r_{i \alpha\alpha}}{2(r_{i \alpha}^2+z_{i \alpha}^2)^{\frac{3}{2}}}+\frac{z_{i \alpha}}{2r_i \sqrt{r_{i \alpha}^2+z_{i \alpha}^2}},\quad \alpha\in [0,1],  \quad i=1,2,3,   \label{H_S}\\[1.5ex]
0&=&z_{i\alpha}z_{i\alpha\alpha}+r_{i\alpha}r_{i\alpha\alpha},\quad \alpha\in [0,1],  \quad i=1,2,3.\label{V_S3}
\end{eqnarray}

From (\ref{const_S}) and the angle conditions, (\ref{6bc})--(\ref{9bc}),
\begin{eqnarray}
r_{1\alpha}(0)=c_1^{1/2}\cos(\theta_c), \quad {z_{2\alpha}}(1)={r_{3\alpha}}(0)=0,\quad\quad \theta_1(0)=\theta_c, \quad \theta_2(1)=0, \quad \theta_3(0)={\pi}/{2},  \label{bc_S3}\\[1ex]
{r_{1 \alpha}}(1)\cdot {r_{3 \alpha}}(1) +
{z_{1 \alpha}}(1)  \cdot {z_{3 \alpha}}(1)=c_1^{1/2} c_3^{1/2} \cos(\beta), \quad\quad - \bar{\theta}_1+ \bar{\theta}_3=\beta, \label{bc_S4}\\[1ex]
 {r_{2 \alpha}}(0) \cdot {r_{3 \alpha}}(1) + {z_{2 \alpha}}(0) \cdot {z_{3 \alpha}}(1)=-c_2^{1/2} c_3^{1/2} \cos(\beta),\quad\quad
\bar{\theta}_2 - \bar{\theta}_3=-\pi+\beta, \label{bc_S5}
\end{eqnarray}
where from (\ref{tjangles}), $\bar{\theta}_1=\theta_1(1),$   $\bar{\theta}_2=\theta_2(0),$ $\bar{\theta}_3=\theta_3(1).$

Lastly we recall our assumption that the curves $\Gamma_i,$ $i=1,2,3,$ do not self-intersect, that their only intersection occurs at the triple junction and that their only contact points with the substrate and the inert cylinder occur at: $(r_1^\ast,0),(r_3^\ast,0)$ and $(1,z_2^\ast)$. Note that (\ref{0r1}) implies that the hole remains open.

\subsection{Analytic Expressions for the Steady State Solutions}\label{analytic solutions}


  From (\ref{conH}), it follows that at steady state in our axi-symmetric geometry,   the two exterior surfaces have constant and equal mean curvatures and that the mean curvature of the grain boundary vanishes.
 In 1841 Delaunay \cite{Del41,EHLM} proved  that the only surfaces of revolution with constant mean curvature were the surfaces obtained by rotating the roulettes of the conics, namely,  planes orthogonal to the axis of symmetry, cylinders, spheres,  catenoids,  unduloids, and nodoids; among these surfaces,  the catenoid and the plane satisfy $H=0$, while the remaining surfaces satisfy $H=const\ne0$. We shall now see that the boundary
 conditions imply that in fact the exterior surfaces are nodoids and the grain boundary surface is a catenoid.


From  (\ref{conH})--(\ref{H_S}), and recalling (\ref{0r1}),(\ref{bc_S1}), and (\ref{const_S}),
\begin{equation} \label{HL0}
\lambda=\frac{r_{i\alpha}z_{i\alpha\alpha}-z_{i\alpha}r_{i\alpha\alpha}}{2(r_{i\alpha}^2+z_{i\alpha}^2)^{\frac{3}{2}}}+\frac{z_{i\alpha}}{2r_i\sqrt{r_{i\alpha}^2+z_{i\alpha}^2}}, \quad i=1,2, \quad
0=\frac{r_{3\alpha}z_{3\alpha\alpha}-z_{3\alpha}r_{3\alpha\alpha}}{2(r_{3\alpha}^2+z_{3\alpha}^2)^{\frac{3}{2}}}+\frac{z_{3\alpha}}{2r_3\sqrt{r_{3\alpha}^2+z_{3\alpha}^2}}, \quad \alpha \in [0,1],
\end{equation}
which may be written as
\begin{equation} \nonumber
 2\lambda r_i r_{i\alpha}=\frac{d}{d\alpha}\left(\frac{r_i z_{i\alpha}}{\sqrt{r_{i\alpha}^2+z_{i\alpha}^2}}\right),\quad  i=1,2,\quad
  0=\frac{d}{d\alpha}\left(\frac{r_3 z_{3\alpha}}{\sqrt{r_{3\alpha}^2+z_{3\alpha}^2}}\right),\quad \alpha \in [0,1].
\end{equation}
Integrating with respect to $\alpha,$ recalling (\ref{0r1}),(\ref{bc_S1}) and dividing by $r_i$,
\begin{equation}\label{su}
\sin(\theta_i)=\frac{z_{i\alpha}}{\sqrt{r_{i\alpha}^2+z_{i\alpha}^2}}=\lambda r_i+\frac{C_i}{r_i},\quad  i=1,2,\quad
\sin(\theta_3)=\frac{z_{3\alpha}}{\sqrt{r_{3\alpha}^2+z_{3\alpha}^2}}=\frac{C_3}{r_3},\quad \alpha \in [0,1],
\end{equation}
where $C_i\in \mathbb{R}$, $i=1,2,3,$ are constants of integration, and   $\theta_i=\theta_i(\alpha)$     corresponds  to the angle description given in (\ref{angles}) for the curve $(r_i(\alpha), z_i(\alpha))$, for $i=1,2,3,$  and $\alpha\in [0,1]$.

Let us first consider the grain boundary, $(r_3(\alpha), z_3(\alpha)),$ in a bit more detail.
From (\ref{bc_S1}), (\ref{bc_S3}), (\ref{su}),
\begin{equation} \label{AC3}
\theta_3(0)=\frac{\pi}{2},\quad \sin(\theta_3(0))= \left.\frac{z_{3\alpha}}{\sqrt{r_{3\alpha}^2+z_{3\alpha}^2}}\right|_{\alpha=0}=1=\frac{C_3}{r_3^\ast},
\end{equation}
and thus $C_3=r_3^\ast.$ Setting $A:=r_3^\ast$ to simplify notation, with $0<A<1$ by (\ref{0r1}), (\ref{su})--(\ref{AC3}) yield that
\begin{equation} \label{AC3a}
 \sin(\theta_3(\alpha))= \frac{z_{3\alpha}(\alpha)}{\sqrt{r_{3\alpha}^2(\alpha)+z_{3\alpha}^2(\alpha)}}=\frac{A}{r_3(\alpha)},\quad \alpha \in [0,1].
\end{equation}
 Noting that (\ref{AC3a}) implies that $A r_3^{-1}(\alpha)\le 1$ for $\alpha \in [0,1],$ we
get using (\ref{0r1}) that
 \begin{equation} \label{cincrease3}
  0<A \le r_{3}(\alpha) <1,\quad z_{3 \alpha}(\alpha) > 0,\quad \alpha\in[0,1],
  \end{equation}
  and hence that $r_{3 \alpha}(\alpha) \ge 0$ in a positive neighborhood of $\alpha =0.$ If $r_{3 \alpha}(\alpha)$ vanishes initially and in
  a positive neighborhood of $\alpha=0,$ then (\ref{const_S}), (\ref{HL0}) imply that $r_{3 \alpha \alpha}(\alpha)>0$ in this neighborhood, which yields
  a contradiction. So $r_{3 \alpha}(\alpha)>0$ in some positive neighborhood of $\alpha =0.$ If it does not hold that $r_{3 \alpha}(\alpha)>0$ for all $\alpha \in
  (0,1]$, then a similar contradiction arises in considering the smallest value of $\alpha$ at which $r_{3 \alpha}(\alpha)$ vanishes.
  Therefore,
  \begin{equation} \label{cincrease4}
   r_{3 \alpha}(\alpha) > 0,\quad \alpha\in(0,1],
  \end{equation}
  and  using  (\ref{AC3a}) we may conclude that
  \begin{equation} \label{cincrease5}
  0< \frac{r_{3\alpha}(\alpha)}{\sqrt{r_{3\alpha}^2(\alpha)+z_{3\alpha}^2(\alpha)}} =\cos(\theta_3(\alpha))=\frac{\sqrt{r_3^2(\alpha)-A^2}}{r_3(\alpha)},\quad \alpha\in(0,1].
  \end{equation}
   Differentiating (\ref{AC3a})  with respect to $\alpha$ and using  (\ref{cincrease5}) yields, moreover, that
  \begin{equation} \label{cincrease6}
   \theta_{3 \alpha}(\alpha) < 0,\quad \alpha\in(0,1].
  \end{equation}

  From (\ref{AC3a}), (\ref{cincrease3}), and (\ref{cincrease5}),
 \begin{equation} \label{star}
\frac{dz_3}{dr_3}=\frac{A}{\sqrt{r_3^2-A^2}}>0.\end{equation}
Since $0<A <r_3(\alpha)$ for $\alpha \in (0,1]$ by (\ref{cincrease3})--(\ref{cincrease4}),
   we may integrate (\ref{star}) with respect to $r_3,$ and  obtain using  (\ref{AC3a}) that
\begin{equation}
r_3=A\,\cosh(\frac{z_3}{A})=\frac{A}{\sin(\theta_3)}, \quad z_3=A\,\acosh(\frac{r_3}{A})=A\,\log\left(\frac{1+\cos(\theta_3)}{\sin(\theta_3)}\right).\label{caten1}
\end{equation}
Finally,  from (\ref{bc_S3}), (\ref{cincrease6}), (\ref{caten1}),
\begin{equation} \label{psi3f}
0<\arcsin(A/\bar{r})=\bar{\theta}_3=\theta_3(1) <\theta_3(\alpha) < \theta_3(0)=\frac{\pi}{2},\quad \alpha\in(0,1),
\end{equation}
where 
\begin{equation}\label{E2}
0< \bar{r}:= r_3(1)<1, \quad 0<\bar{z}:=z_3(1)=A\,\acosh(\bar{r}/A).
\end{equation}
From (\ref{star}) and the discussion above,
\begin{lemma} \label{lemma_G3}
Given any $A, \bar{r} \in \mathbb{R},$
\begin{equation} \label{C_G3}
0<A<\bar{r}<1, \end{equation}
a unique solution curve $\Gamma_3$ is defined by (\ref{caten1})--(\ref{E2}) which
satisfies the conditions in (\ref{0r1})--(\ref{bc_S3}) which pertain to $\Gamma_3$, with  $r_3^\ast=A$, $\bar{z}=A\,\acosh(\bar{r}/A)$, $\bar{\theta}_3=\arcsin(A/\bar{r});$  it is monotone increasing and hence non-self-intersecting. Moreover, any solution curve $\Gamma_3$ satisfying
the conditions in (\ref{0r1})--(\ref{bc_S3}) which pertain to $\Gamma_3$, with $r_3(0)=A$ and $r_3(1)=\bar{r}$, satisfies (\ref{C_G3}) and
is monotone increasing and non-self-intersecting.
\end{lemma}
\noindent Note that (\ref{caten1})--(\ref{E2}) imply that $\Gamma_3$ is smooth and of finite  length, so the parametrization conditions (\ref{const_S}), (\ref{V_S3}) are readily satisfied.

 In order to obtain a steady state solution,
it remains to determine $\Gamma_1,$  $\Gamma_2,$ and to fit them together with $\Gamma_3$ in accordance with the angle conditions (\ref{bc_S4})--(\ref{bc_S5}), and to check the intersections and self-intersections.

\bigskip
Let us now consider $\Gamma_1,$ the inner exterior surface, in greater detail.
Recalling (\ref{0r1})--(\ref{bc_S1}),   we  require that
\begin{equation} \label{sst1}
0<r_1(\alpha)<1,\quad z_1(\alpha)>0, \quad \alpha \in (0,1),
\end{equation}
\begin{equation}\label{sst11}
(r_1,z_1)(0)=(r_1^\ast,0), \quad (r_1,z_1)(1)=(\bar{r},\bar{z}), \quad 0 < r_1^\ast,\, \bar{r}<1,\quad \bar{z}>0.
\end{equation}
Note that (\ref{sst11})  implies  (\ref{endsne}). \color{black}
 From (\ref{thetac}), (\ref{gammagamma}), (\ref{const_S}), (\ref{bc_S3})--(\ref{bc_S4}), and (\ref{psi3f}), we obtain the following  conditions
for $\beta\in [\pi/2,\, \pi),$ $\theta_c \in [0,\, \pi]$,
\begin{equation} \label{c10}
\theta_1(0)=\theta_c, \quad \left.\frac{z_{1\alpha}}{\sqrt{r_{1\alpha}^2+z_{1\alpha}^2}}\right|_{\alpha=0}=\sin(\theta_c)\geq0,\end{equation}
\begin{equation}\label{l42}
 \theta_1(1)=\bar{\theta}_1=\bar{\theta}_3 -\beta,\quad -\pi < \bar{\theta}_1 <0, \quad  \left.\frac{z_{1\alpha}}{\sqrt{r_{1\alpha}^2+z_{1\alpha}^2}}\right|_{\alpha=1}=\sin(\bar{\theta}_1) < 0,
 \end{equation}
and by (\ref{su}),
 \begin{equation}\label{su1}
\sin(\theta_1)=\frac{z_{1\alpha}}{\sqrt{r_{1\alpha}^2+z_{1\alpha}^2}}=\lambda r_1+\frac{C_1}{r_1}, \quad  \alpha \in [0,1].
\end{equation}

\begin{lemma} \label{lemma_al}
For $(r_1(\alpha), z_1(\alpha)),$ $\alpha \in [0,1],$  satisfying (\ref{sst1})--(\ref{su1}), there exists a unique  $r_1=a_\ell \in(0,1)$ such that $z_{1\alpha}=0$ when $r_1=a_\ell$.
\end{lemma}

\textit{Proof.} It follows from (\ref{c10}), (\ref{l42}) that
 there exists $\alpha_0 \in[0,1)$   \color{black} for which $z_{1\alpha}(\alpha_0)=0$. 
  By (\ref{su1}),
\begin{equation}
\lambda r_1(\alpha_0)+\frac{C_1}{r_1(\alpha_0)}=0.\label{bal}
\end{equation}
Recalling  (\ref{sst1}), (\ref{sst11}) \color{black}  it follows  that there is a unique value $r_1(\alpha_0)  \in (0,1),$ which we denote by $a_\ell$, which  satisfies (\ref{bal}).
 \hfill$\square$

\bigskip

Setting $r_1(\alpha_0)=a_\ell$ in  (\ref{bal}),  we get that $C_1=-\lambda a_\ell^2,$ and using this result in (\ref{su1}) yields that
\begin{equation}
\sin(\theta_1)=
\frac{z_{1\alpha}}{\sqrt{r_{1\alpha}^2+z_{1\alpha}^2}}=\lambda (r_1-\frac{a_\ell^2}{r_1}), \quad \alpha\in[0,1].\label{f1}
\end{equation}
Note that given $\bar{\theta}_3$, $\bar{r}$, $\bar{z}$, then $\bar{\theta}_1$, ${r_1}(1)$, ${z_1}(1)$ are known from
(\ref{sst11}), (\ref{l42}). But further information regarding $\lambda$ and $a_\ell$ is needed to complete the description of
$\Gamma_1$. In particular, it remains to verify that
 (\ref{sst11})--(\ref{c10}) can be satisfied, namely there exists $r_1^\ast\in (0,1)$ such that $(r_1(0),z_1(0))=(r_1^\ast, 0)$ with
 $\theta_1(0)=\theta_c$, and to consider the possible intersections and self-intersections.

 \smallskip
 Before proceeding further with regard to $\Gamma_1,$ let us first turn
to consider $\Gamma_2$, the outer exterior surface.
Recalling (\ref{0r1})--(\ref{bc_S1}),   we  obtain the conditions
\begin{equation} \label{sst2}
0<r_2(\alpha)<1,\quad z_2(\alpha)>0, \quad \alpha \in [0,1),
\end{equation}
\begin{equation}\label{sst22}
(r_2,z_2)(0)=(\bar{r},\bar{z}), \quad (r_2,z_2)(1)=(1,{z_2^\ast}), \quad 0 < z_2^\ast,
\end{equation}
and (\ref{sst22})  implies  (\ref{endsne}).
From (\ref{bc_S3})--(\ref{bc_S5}),  we obtain the conditions,
\begin{equation}\label{l22}
 \theta_2(0)=\bar{\theta}_2=\bar{\theta}_3 -\pi +\beta=\bar{\theta}_1 -\pi +2\beta,\quad \left.\frac{z_{2\alpha}}{\sqrt{r_{2\alpha}^2+z_{2\alpha}^2}}\right|_{\alpha=0}=\sin(\bar{\theta}_2),
 \end{equation}
\begin{equation} \label{c22}
\theta_2(1)= 0,\quad z_{2\alpha}(1)=0.\end{equation}

\bigskip
From (\ref{su}), (\ref{sst22}), (\ref{c22}),
\begin{equation} \nonumber
\sin(\theta_2(1))=0=\lambda+ {C_2}.\end{equation}
Hence $C_2=-\lambda,$ and using (\ref{su}) we get that
\begin{equation}
\sin(\theta_2)=\frac{z_{2\alpha}}{\sqrt{r_{2\alpha}^2+z_{2\alpha}^2}}=\lambda (r_2-\frac{1}{r_2}), \quad \alpha\in[0,1].\label{f2}
\end{equation}

\smallskip
Looking at (\ref{l22})--(\ref{f2}), we see that $\lambda$, $\bar{r}$ and $\bar{z}$ uniquely determine $\Gamma_2$. We now prove

\begin{lemma} \label{lambda_neg}
{The mean curvature, $\lambda,$ of the exterior surfaces  $(r_i(\alpha), z_i(\alpha)),$ $i=1,2$,  satisfies $\lambda<0$.}\end{lemma}

\textit{Proof.}
Let us first suppose that $\lambda=0.$ As noted above \cite{Del41,EHLM}, the only surfaces of revolution with vanishing mean curvature are  planes  and  catenoids.
From (\ref{bc_S1})--(\ref{r0z0}), it follows that the inner exterior surface is not  a plane orthogonal to the axis of symmetry.
From (\ref{0z})--(\ref{bc_S1}), $z_1(\alpha)\ge 0,$ $\alpha \in [0,1]$, and  $\bar{\theta}_1\notin\left(0,\frac{\pi}{2}\right)$ by  (\ref{l42}); hence the inner exterior surface is not a catenoid. Therefore $\lambda\neq0.$

 Let us now suppose that $\lambda>0$.   Since $r_1>0$, the right hand side of (\ref{f1}) is a strictly increasing function of $r_1$
  for  $\alpha \in [0,1]$. It  follows from (\ref{bc_S1}), (\ref{r0z0}), (\ref{cincrease4}) that $r_1=r_1(\alpha)$ satisfies
  $$r_1(0)<A=r_3(0)<r_3(1)=\bar{r}=r_1(1),$$
   However, by   (\ref{c10}), (\ref{l42}),
  $$\left.\frac{z_{1\alpha}}{\sqrt{r_{1\alpha}^2+z_{1\alpha}^2}}\right|_{\alpha=0}=\sin(\theta_c)\geq0 > \sin(\bar{\theta}_1)=\left.\frac{z_{1\alpha}}{\sqrt{r_{1\alpha}^2+z_{1\alpha}^2}}\right|_{\alpha=1},$$
  which yields a contradiction to (\ref{f1}).
Hence  $\lambda<0.$ \hfill$\square$

\begin{note} \label{G2G3} Since $\lambda<0$, it now follows from (\ref{AC3a}) and  (\ref{f2}) that $z_{3\alpha}(\alpha), z_{2\alpha}(\alpha)>0$
for $\alpha \in [0,1)$, and by  (\ref{bc_S1}),  $(r_3,z_3)(1)=(\bar{r},\bar{z})=(r_2,z_2)(0)$ . Hence $\Gamma_2$ has no self-intersections,
and $\Gamma_2 \cap \Gamma_3=(\bar{r}, \bar{z}).$
\end{note}

%
\color{black}

\begin{lemma}\label{monotone}
The functions $\theta_i(\alpha),$ $i=1,2$, are strictly decreasing on the interval $\alpha\in[0,1]$.
More specifically, $\theta_{i \alpha}(\alpha)<0,$ $i=1,2$, for $\alpha\in[0,1]$.
\end{lemma}

\textit{Proof.} Differentiating (\ref{f1}), (\ref{f2}), we obtain that
\begin{equation}
\frac{d}{d\alpha}\left(\frac{z_{1\alpha}}{\sqrt{r_{1\alpha}^2+z_{1\alpha}^2}}\right)=r_{1\alpha}\lambda\left(1+\frac{a_\ell^2}{r_1^2}\right),\quad
\frac{d}{d\alpha}\left(\frac{z_{2\alpha}}{\sqrt{r_{2\alpha}^2+z_{2\alpha}^2}}\right)=r_{2\alpha}\lambda\left(1+\frac{1}{r_2^2}\right),
 \quad \alpha \in [0,1],\label{lem51}
\end{equation}
  and using (\ref{const_S}),
\begin{equation}\label{lem53}
\frac{d}{d\alpha}\left(\frac{z_{i\alpha}}{\sqrt{r_{i\alpha}^2+z_{i\alpha}^2}}\right)=
\frac{d\sin(\theta_i)}{d\alpha}=\cos(\theta_i)\theta_{i\alpha}={r_{i\alpha}\theta_{i\alpha}}c_i^{-1/2},\ i=1,2, \quad \alpha \in [0,1].
\end{equation}
Since $\lambda<0$ by Lemma \ref{lambda_neg}, we obtain from (\ref{lem51})--(\ref{lem53}) that for $\alpha \in [0,1]$, if $r_{i\alpha}\ne 0$, then
\begin{eqnarray}
\theta_{1\alpha}=\lambda c_1^{1/2} \left(1+\frac{a_\ell^2}{r_1^2}\right)<\lambda c_1^{1/2}<0, \quad
\theta_{2\alpha}=\lambda c_2^{1/2} \left(1+\frac{1}{r_2^2}\right)<\lambda c_2^{1/2}<0.
\end{eqnarray}
Noting  that if $r_{i\alpha}= 0$, then $\cos(\theta_i)=0,$ which occurs only at isolated values of $\theta_i$, and
recalling  our regularity assumptions, the lemma follows. \hfill$\square$

\bigskip
 Recalling (\ref{c10}),  (\ref{c22}),  Lemma \ref{monotone} implies the following
\begin{cor} \label{anglebounds} $\bar{\theta}_1<\theta_c,$  $0<\bar{\theta}_2.$
\end{cor}

\begin{lemma} \label{angle_order} There exist no steady state solutions when $\beta=\pi/2$.
\end{lemma}

\textit{Proof.}
From Corollary \ref{anglebounds} and  (\ref{psi3f}), (\ref{l42}), (\ref{l22}), we get that
\begin{equation}\label{E4}
-\pi < \bar{\theta}_1 < 0<\bar{\theta}_2< \bar{\theta}_3 < \frac{\pi}{2}.
\end{equation}
However, if $\beta=\pi/2,$ then $\bar{\theta}_1=\bar{\theta}_2=\bar{\theta}_3 - \pi/2$ by (\ref{l22}), which contradicts (\ref{E4}).
\hfill$\square$

\begin{lemma}\label{zerothetac}
There exist no steady state solutions when $\theta_c=0$.
\end{lemma}

\textit{Proof.} Suppose that $\theta_c=0.$ Then $\sin(\theta_1(0))=\sin(\theta_c)=0,$ and from Lemma \ref{monotone} and (\ref{c10})-(\ref{su1})
we get that $z_{1\alpha}(\alpha) \le 0,$ $\alpha \in [0,1],$ yielding a contradiction in (\ref{bc_S1})--(\ref{r0z0}).
\hfill$\square$

\bigskip
Since $\theta_i(\alpha),$  $i=1,2$ are  strictly decreasing functions on the interval $\alpha\in[0,1]$,
we may make a change of variables and use the variable $\theta_i$ for the exterior surface $(r_i, z_i)$ for $i=1,2$, respectively, instead of  $\alpha$.
Explicit expressions for $(r_i(\theta_i), z_i(\theta_i))$, $i=1,2,$ are derived below. We recall that it was shown in   (\ref{cincrease6}) that $\theta_3(\alpha)$ is a strictly
decreasing function of $\alpha$ for $\alpha\in [0,1]$, and explicit formulas for $(r_3(\theta_3), z_3(\theta_3))$ were given in (\ref{caten1}).
Thus, in particular,
\begin{equation} \label{E5}
\theta_1\in[\bar{\theta}_1,\theta_c], \quad \theta_2\in[0,\bar{\theta}_2], \quad \theta_3 \in [\bar{\theta}_3, \pi/2],
\end{equation}
where $\bar{\theta}_1,$ $\bar{\theta}_2,$ $\bar{\theta}_3$ satisfy (\ref{E4}), and in view of Lemma  \ref{angle_order} and Lemma \ref{zerothetac}, we
henceforth assume that
\begin{equation} \label{obetatheta}
\beta \in (\pi/2,\, \pi), \quad \theta_c \in (0, \, \pi].
\end{equation}

\bigskip
To find $r_1(\theta_1),\ r_2(\theta_2),$ since $\lambda<0$ by Lemma \ref{lambda_neg}, we may write (\ref{f1}), (\ref{f2}), respectively, as
\begin{equation}\nonumber
r_1^2-\frac{r_1\sin(\theta_1)}{\lambda}-a_\ell^2=0,\quad
r_2^2-\frac{r_2\sin(\theta_2)}{\lambda}-1=0,
\end{equation}
then solving the above equations for $r_1,r_2>0$ yields
\begin{equation}
r_1(\theta_1)=\frac{\sin(\theta_1)-\sqrt{\sin^2(\theta_1)+4\lambda^2a_\ell^2}}{2\lambda},\quad \theta_1\in[\bar{\theta}_1,\theta_c],  \quad
r_2(\theta_2)=\frac{\sin(\theta_2)-\sqrt{\sin^2(\theta_2)+4\lambda^2}}{2\lambda},\quad \theta_2\in[0,\bar{\theta}_2].\label{rnod}
\end{equation}

Noting that ${z_{i\theta_i}}=\tan(\theta_i){r_{i\theta_i}}$, $i=1,2,$ we obtain from  (\ref{rnod}) that
\begin{equation}\nonumber
z_{1\theta_1}=\frac{1}{2\lambda}\left(\sin(\theta_1)-\frac{\sin^2(\theta_1)}{\sqrt{\sin^2(\theta_1)+4\lambda^2a_\ell^2}}\right),\quad \theta_1\in[\bar{\theta}_1,\theta_c], \quad
z_{2\theta_2}=\frac{1}{2\lambda}\left(\sin(\theta_2)-\frac{\sin^2(\theta_2)}{\sqrt{\sin^2(\theta_2)+4\lambda^2}}\right),\quad\theta_2\in[0,\bar{\theta}_2].
\end{equation}
Solving these equations and taking  (\ref{bc_S1}) into account, we obtain that
\begin{align}
\begin{array}{l}
z_1(\theta_1)=-\frac{1}{2\lambda}\int_{\bar{\theta}_1}^{\theta_1}\left(\frac{\sin^2(x)}{\sqrt{\sin^2(x)+4\lambda^2a_\ell^2}}-\sin(x)\right)dx+\bar{z}, \quad \theta_1\in[\bar{\theta}_1,\theta_c],\\
z_2(\theta_2)=-\frac{1}{2\lambda}\int_{\bar{\theta}_2}^{\theta_2}\left(\frac{\sin^2(x)}{\sqrt{\sin^2(x)+4\lambda^2}}-\sin(x)\right)dx+\bar{z},\quad \theta_2\in[0,\bar{\theta}_2].\label{znod}
\end{array}
\end{align}
The curves in (\ref{rnod})--(\ref{znod})  represent the meridian profiles of nodoids \cite{EHLM},\cite{REFF} with equal mean curvatures,  $\lambda<0$.

\bigskip
Based on our results so far, we may state the following: given $\beta\in (\frac{\pi}{2},\pi)$, $\theta_c\in(0,\pi]$,
if there exists a steady state solution to (\ref{meridian1})--(\ref{ic}), then it may be expressed in terms of the angle variables as
\begin{equation}
\begin{cases}\label{solutions}
&\begin{pmatrix}r_1(\theta_1)\\z_1(\theta_1)\end{pmatrix}=
\begin{pmatrix}\frac{\sin(\theta_1)-\sqrt{\sin^2(\theta_1)+4\lambda^2a_\ell^2}}{2\lambda}\\\frac{-1}{2\lambda}\int_{\bar{\theta}_1}^{\theta_1}{\left(\frac{\sin^2(x)}{\sqrt{\sin^2(x)+4\lambda^2a_\ell^2}}-\sin(x)\right)dx}+\bar{z}\end{pmatrix},\quad \theta_1\in\left[\bar{\theta}_1,\theta_c\right], \\
\\
&\begin{pmatrix}r_2(\theta_2)\\z_2(\theta_2)\end{pmatrix}=\begin{pmatrix}\frac{\sin(\theta_2)-\sqrt{\sin^2(\theta_2)+4\lambda^2}}{2\lambda}\\\frac{-1}{2\lambda}\int_{\bar{\theta}_2}^{\theta_2}\left(\frac{\sin^2(x)}{\sqrt{\sin^2(x)+4\lambda^2}}-\sin(x)\right)dx+\bar{z}\end{pmatrix},\quad \theta_2\in\left[0,\bar{\theta}_2\right],\\
\\
&\begin{pmatrix}r_3(\theta_3)\\z_3(\theta_3)\end{pmatrix}=\begin{pmatrix}\frac{A}{\sin(\theta_3)}\\A\,\log\left(\frac{1+\cos(\theta_3)}{\sin(\theta_3)}\right)\end{pmatrix},\quad \theta_3\in\left[\bar{\theta}_3,\frac{\pi}{2}\right],\\
\\
&\bar{\theta}_1= \bar{\theta}_3-\beta,\quad\bar{\theta}_2 = \bar{\theta}_3-\pi+\beta,\quad \bar{\theta}_3=\arcsin(A/\bar{r}), \quad\bar{z} = A\,\acosh(\bar{r}/A),
\end{cases}
\end{equation}
with the range conventions: $\acosh: [1,\infty) \rightarrow [0,\infty)$, $\arcsin: [-1,1]\rightarrow [-\pi/2, \pi/2],$ and where \color{black} $\bar{\theta}_1,\,\bar{\theta}_2,\,\bar{\theta}_3$ satisfy (\ref{E4}). 
\par The problem formulated in (\ref{meridian1})--(\ref{l4a}), to be satisfied by the steady states, contains four (unknown) endpoints: $(\bar{r}(t), \bar{z}(t)),$  $({r_1^\ast}(t), 0),$ $(1, {z_2^\ast}(t)),$ $({r_3^\ast}(t), 0),$ described by $5$ unknown functions,
in addition to the $2$ physical parameters $(\beta, \,\theta_c)$. The expressions in (\ref{solutions}) for the steady state solutions contain $4$ parameters:  $\bar{r},$
 $A,$   $\lambda$,  $a_\ell,$ \color{black} in addition to  $(\beta,\, \theta_c).$
    Thus we begin with a problem formulation containing $5$ unknown functions, whose dependence on  the $2$ physical parameters, $(\beta,\, \theta_c)$,
     is not  obvious, and  obtain the expressions given by (\ref{solutions}), for the steady state solutions, which  contain $4$ \color{black} unknown parameters,
     whose dependence on   $(\beta,\, \theta_c)$ is not transparent.
     Thus \color{black}  (\ref{solutions}) can be viewed as necessary conditions to be satisfied by steady state solutions. Note that the initial conditions given in (\ref{ic}) have not been considered in deriving (\ref{solutions}). In particular, the total volume  at $t=0$, $\mathbb{V}(0)$, which is given by (\ref{volume}), and the initial total free energy, $\mathbb{E}(0)$, which is indicated in (\ref{etot}) have not been taken into account; in particular, the volume constraint given in Lemma \ref{dvdt} and the decrease (non-increase)  of the total free energy given in Lemma \ref{dedt} have not been considered.  
     Our approach is to explore the set of all possible solutions, leaving the restrictions implied by Lemma \ref{dedt} and Lemma \ref{dvdt} to be considered later. Furthermore, we note that there was some freedom in choosing the $4$ parameters $A,\bar{r},\lambda,a_\ell$, for example we could have included $\bar{z}$ in designating the unknown parameters, rather than prescribing it explicitly in (\ref{solutions}) as we did above. Shortly, however, in Subsection \ref{parameters and constraints} we demonstrate  a reduction to two parametric degrees of freedom, in addition to the physical parameters $\beta,\ \theta_c$,
     and the constraints implied by volume conservation and energy decay, which are discussed briefly further in Subsection \ref{energy volume_ss} and in Section \ref{numerical solutions}.\\
     \par In obtaining (\ref{solutions}),
various parametric constraints were encountered, which need to be satisfied in order to obtain a necessary and sufficient prescription for the possible steady states.  These constraints are now formulated and analyzed.  Let us first note that \color{black}
       \begin{equation} \label{constraints1} 0<A<\bar{r}<1, \quad \lambda <0, \quad 0 < a_{\ell}<1,\quad\quad \rm{(C1)-(C3)}\end{equation}
                 are implied by  Lemma \ref{lemma_G3},   Lemma \ref{lemma_al} and Lemma \ref{lambda_neg}.
     Let us now assume the parameters appearing in (\ref{solutions}) to be given and to satisfy (\ref{solutions})--(\ref{constraints1}), and we obtain necessary conditions that need to be satisfied in order to guarantee that (\ref{solutions}) satisfies \color{black}  (\ref{meridian1})--(\ref{l4a}) for some prescribed choice of  $(\beta,\, \theta_c),$  taking into consideration questions concerning self-intersections and intersections which were postponed earlier.

Let us start by examining \color{black} $\Gamma_3$. Note that $\Gamma_3$  as prescribed in (\ref{solutions}) depends on $A,$  $\bar{r}$, and we are assuming  $A,$  $\bar{r}$ to  satisfy  (\ref{constraints1}). Setting
\begin{equation}
\bar{z}=A\,\acosh\Bigl(\frac{\bar{r}}{A}\Bigr),\quad \bar{z}>0, \quad\quad \bar{\theta}_3=\arcsin(A/\bar{r}),\quad \bar{\theta}_3\in (0,\pi/2),\label{bar3}\end{equation}
in accordance with  (\ref{solutions})--(\ref{constraints1}), it is easy to check that $(r_3(\bar{\theta}_3), z_3(\bar{\theta}_3)) = (\bar{r}, \bar{z}).$ Moreover, since $(r_3(\pi/2), z_3(\pi/2))=(A,0)$, we identify 
\begin{equation}
r_3^\ast=A,
\end{equation} thus obtaining that (\ref{r0z0}) holds for $r_3^\ast$, $\bar{r},$ $\bar{z},$  and
  $(r_3(\pi/2), z_3(\pi/2)) \neq (r_3(\bar{\theta}_3), z_3(\bar{\theta}_3)).$  It is easy to check that  $r_{3 \theta_3}(\theta_3)<0, z_{3 \theta_3}(\theta_3) <0$ \color{black} for $\theta_3 \in (\bar{\theta}_3, \pi/2),$  and $|\Gamma_3|=\int_{\bar{\theta}_3}^{\pi/2} \sqrt{ r_{3 \theta_3}^2(\theta_3) + z_{3 \theta_3}^2(\theta_3)}\, d\theta_3   <\infty.$ So  $\Gamma_3$ can be readily reparametrized by $\alpha \in [0,1]$ in accordance with (\ref{bc_S1}), (\ref{const_S}), (\ref{V_S3}), (\ref{bc_S3}). Following  reparametrization,   $r_{3 \alpha}(\alpha), z_{3 \alpha}(\alpha) >0$ for $\alpha \in (0,1)$, and hence $\Gamma_3$ is non-self-intersecting, $\Gamma_3$ intersects the substrate only at $(r_3^\ast ,0)$, $\Gamma_3$  does not \color{black} intersect the inert cylinder, \color{black} 
 and  the monotonicity of  $(r_3(\alpha), z_3(\alpha))$  implies that (\ref{0r1})-(\ref{0z}), (\ref{endsne})   hold for $\Gamma_3$.
 Also, (\ref{conH})--(\ref{H_S}) hold, as can be verified directly.

  In summary\color{black}, $\Gamma_3$ represents a catenoid, as it describes the meridian profile of an axi-symmetric surface with mean curvature zero which is not planar. Furthermore, (\ref{solutions})--(\ref{constraints1}) can be seen to imply (\ref{AC3a})--(\ref{cincrease6}), in particular \color{black} 
 $$r_{3 \alpha}(\alpha), \, z_{3 \alpha}(\alpha) >0, \quad \theta_{3 \alpha}(\alpha) <0, \quad \alpha \in (0,1],$$
 from which we may conclude that $\Gamma_3$ is bounded from above and at either side, respectively, by the  lines
 \begin{equation} \label{lines3}
 z-\bar{z} = \tan(\bar{\theta}_3)(r-\bar{r}), \quad r=A, \quad r=\bar{r}. \end{equation}

  Next let us examine $\Gamma_2$, which  can be seen from (\ref{solutions}) to depend on $\lambda,$  $A$,  $\bar{r}$, and \color{black}$\beta$. We continue to assume  that (\ref{constraints1})  holds.
  In accordance with (\ref{bc_S1}), (\ref{solutions}), we impose the  constraint
  \begin{equation} \label{constraints2}
   \bar{r}=r_2(\bar{\theta}_2)=\frac{1}{2\lambda} \Biggl( \sin(\bar{\theta}_2)-\sqrt{\sin^2(\bar{\theta}_2)+4\lambda^2}\Biggr).\quad\quad \rm{(C4)}
  \end{equation}
  From the expression for $\bar{\theta}_2$ given in (\ref{solutions}), since $\beta\in\left(\frac{\pi}{2},\pi\right),$ one obtains that  $\bar{\theta}_2\in\left(\bar{\theta}_3-\frac{\pi}{2},\bar{\theta}_3\right)\subset\left(-\frac{\pi}{2},\frac{\pi}{2}\right).$ Thus, since $0<\bar{r}<1,\,\lambda<0$  by (\ref{constraints1}), equation (\ref{constraints2}) implies that $\sin(\bar{\theta}_2)>0$. Hence \begin{equation}0<\bar{\theta}_2=\bar{\theta}_3-\pi+\beta<\bar{\theta}_3.\label{bar2}\end{equation}
  Looking at (\ref{solutions}), and since $\lambda<0$ by (\ref{constraints1}), we note that $r_2(\theta_2 =0)=1$ and $z_2(\bar{\theta}_2)=\bar{z}$, and  we make the identification
  \begin{equation} \label{zstar}
   z_2^\ast=z_2(\theta_2=0)=\frac{-1}{2\lambda}\int_{\bar{\theta}_2}^{\theta_2=0}\Biggl[\frac{\sin^2(x)}{\sqrt{\sin^2(x)+4\lambda^2}}-\sin(x)\Biggr]dx
  +\bar{z},
  \end{equation}
  which implies that  $(r_2(\theta_2=0), z_2(\theta_2=0))=(1,z_2^\ast),$   $(r_2(\bar{\theta}_2), z_2(\bar{\theta}_2)) = (\bar{r}, \bar{z}).$
 Since by (\ref{constraints1}),(\ref{bar3}), and (\ref{bar2}) we have $0<\bar{r}<1,$ $\lambda<0,$ $\bar{z}>0$, and  $0<\bar{\theta}_2<\frac{\pi}{2}$, we obtain by (\ref{zstar}) that $z_2^\ast>0$ in accordance with (\ref{r0z0}), \color{black}
  and we obtain that $(1,z_2^\ast) \neq (\bar{r}, \bar{z}).$  In analogy with the results for $\Gamma_3$,
  it is easy to check that  $r_{2 \theta_2}(\theta_2), z_{2 \theta_2}(\theta_2) <0$ for $\theta_2 \in (0,\, \bar{\theta}_2),$  and $|\Gamma_2|   <\infty.$ So  $\Gamma_2$ can be readily reparametrized by $\alpha \in [0,1]$ in accordance with (\ref{bc_S1}), (\ref{endsne}), (\ref{const_S}), (\ref{V_S3}), (\ref{bc_S3}), (\ref{bc_S5}), with $(r_2(0),\, z_2(0))=(r_3(1),\, z_3(1))$. Following the reparametrization,   $r_{2 \alpha}(\alpha), z_{2 \alpha}(\alpha) >0$ for $\alpha \in (0,1)$, so (\ref{0r1})-(\ref{0z}) hold, \color{black} $\Gamma_2$ is non-self-intersecting, $\Gamma_2 \cap \Gamma_3
  =(\bar{r},\, \bar{z}),$ $\Gamma_2$ doesn't intersect the substrate and $\Gamma_2$ intersects the inert cylinder only at $(1,z_2^\ast)$. \color{black} 
Furthermore $H_2(\alpha)\equiv \lambda$  in accordance with (\ref{conH})--(\ref{H_S}), as can be verified directly, 
  and $\Gamma_2$ describes the meridian profile of a nodoid
 with   constant negative mean curvature  \cite{Del41,EHLM}. Moreover, (\ref{solutions})--(\ref{constraints1}) can be seen to imply (\ref{f2}), and subsequently,
 $r_{2 \alpha}(\alpha)>0,$ $z_{2 \alpha}(\alpha)>0,$ $\theta_{2 \alpha}(\alpha)<0,$ $\alpha \in (0,1),$
 from which we may conclude that $\Gamma_2$ is bounded from above and at either side, respectively, by the  lines \color{black}
 \begin{equation} \label{lines2}
 z-\bar{z} = \tan(\bar{\theta}_3)(r-\bar{r}), \quad  r=\bar{r}, \quad r=1. \end{equation}

Lastly let us examine $\Gamma_1$, which  by (\ref{solutions}) depends on  $A,$ $a_{\ell},$ $\lambda,$ $\bar{r}$,  and \color{black} $\beta$.
Recalling the expression for $\bar{\theta}_1$ given by (\ref{solutions}), from (\ref{bar3}) and since  $\beta\in\left(\frac{\pi}{2},\pi\right),\,\theta_c \in (0, \pi]$, it follows that
\begin{equation} \label{angles1}
-\pi < \bar{\theta}_1=\bar{\theta}_3-\beta < 0 < \theta_c \le  \pi.
\end{equation}
In particular, from (\ref{bar3}), (\ref{bar2}), and (\ref{angles1}), we obtain that (\ref{E4}) is satisfied.\par
 We get from (\ref{solutions}) that $z_1(\bar{\theta}_1)=\bar{z},$
we need to impose the constraints
\begin{equation} \label{constraints3}
    r_1(\bar{\theta}_1)=\bar{r}, \quad  z_1(\theta_c)=0,\quad\quad \rm{(C5)-(C6)}
\end{equation}
  and  we make the identification
\begin{equation} \label{r1star}
r_1^\ast=r_1(\theta_c).
\end{equation}
Note that prior to (\ref{angles1})--(\ref{r1star}), the physical variable $\theta_c$ had only previously appeared in this subsection as the upper bound on the range of $\theta_1$ in (\ref{solutions}). As we shall see in Subsection \ref{parameters and constraints}\color{black}, most of the parametric constraints can be formulated without reference to $\theta_c.$ 
Since \color{black}$\theta_c \in (0, \pi]$, it follows from (\ref{solutions}), (\ref{constraints1}), that $0 < r_1^\ast <1.$ Thus if (\ref{constraints3}) is satisfied,
then  $(r_1(\theta_c), z_1(\theta_c)) \ne (r_1(\bar{\theta}_1), z_1(\bar{\theta}_1)),$ and hence  (\ref{endsne}) holds.
Note that  $(\bar{r},\, \bar{z})$ and $\bar{\theta}_3$ are determined by  $\Gamma_3,$  and the condition  $r_1(\bar{\theta}_1)=\bar{r}$ given by (\ref{constraints3})
implies an algebraic constraint which is easily handled. However the constraint on \color{black} $z_1(\theta_c)$ which appears in (\ref{constraints3}) is
more problematic, and an understanding with regard to the geometry of $\Gamma_1$ is helpful in this context.

In considering the geometry of $\Gamma_1$, it is useful
to consider its natural  extension, $\widetilde{\Gamma}_1:=\{ (\tilde{r}_1(\theta_1),\, \tilde{z}_1(\theta_1)) \, |\, -\pi \le \theta_1 \le \pi \}$, where
\begin{equation}\label{extG1}
\begin{pmatrix}\tilde{r}_1(\theta_1)\\\tilde{z}_1(\theta_1)\end{pmatrix}=
\begin{pmatrix}\frac{\sin(\theta_1)-\sqrt{\sin^2(\theta_1)+4\lambda^2a_\ell^2}}{2\lambda}\\
\frac{-1}{2\lambda}\int_{\bar{\theta}_1}^{\theta_1}{\left(\frac{\sin^2(x)}{\sqrt{\sin^2(x)+4\lambda^2a_\ell^2}}-\sin(x)\right)dx}+\bar{z}\end{pmatrix}, \quad \theta_1\in [-\pi,\,\pi],
\end{equation}
see \cite{MScKG} Fig. 1. Recalling  (\ref{constraints1}),  we get from (\ref{extG1}) that
\begin{eqnarray}
&\tilde{r}_{1 \theta_1}(\theta_1)<0, \quad \theta_1 \in (-\pi/2, \, \pi/2), \quad \tilde{r}_{1 \theta_1}(\pm \pi/2)=0, \quad \tilde{r}_{1 \theta_1}(\theta_1)>0, \quad \theta_1 \in [-\pi, -\pi/2)\cup(\pi/2, \pi], \label{r1d} \\[1ex]
&\tilde{z}_{1\,\theta_1}(\theta_1) >0, \quad
\theta_1 \in (-\pi, \,  0), \quad \tilde{z}_{1 \theta_1}(0)=\tilde{z}_{1 \theta_1}(\pm \pi)=0, \quad \tilde{z}_{1 \theta_1}(\theta_1)< 0, \quad \theta_1 \in (0, \, \pi).\label{z1d}
\end{eqnarray}
These inequalities  readily imply that  the curve $\widetilde{\Gamma}_1$ is topologically equivalent to  two half circles  joined at $\theta_1=0$ and aligned along the line $\tilde{r}_1 = a_{\ell}.$
It also easily follows from (\ref{constraints1}), (\ref{extG1}) that the right hand side is larger than the left; namely, $\tilde{r}_1(-\pi/2)-\tilde{r}_1(0)>\tilde{r}_1(0)-\tilde{r}_1(\pi/2) $ and \color{black} $\tilde{z}_1(0)-\tilde{z}_1(-\pi) > \tilde{z}_1(0)-\tilde{z}_1(\pi)$.
Using the formulas in (\ref{extG1}), 
we get that $\int_{-\pi}^{\pi} \sqrt{ \tilde{r}_{1 \theta_1}^2(\theta_1) + \tilde{z}_{1 \theta_1}^2(\theta_1)}\, d\theta_1   <\infty,$
and hence $\Gamma_1$, which constitutes a subset of $\widetilde{\Gamma}_1$, can be readily reparametrized by $\alpha \in [0,1]$ in accordance with (\ref{bc_S1})\color{black}, (\ref{const_S}),
(\ref{V_S3})--(\ref{bc_S4}).
As mentioned earlier, $\Gamma_1$ describes a portion of a nodary curve, and (\ref{conH}), (\ref{H_S}) can be checked directly to hold from the formulas in (\ref{extG1}) and the parametrization with respect to $\alpha$.
Moreover, (\ref{solutions})--(\ref{constraints1}) can be seen to imply (\ref{f1}), and subsequently,
 $\theta_{1 \alpha}(\alpha)<0,$ $\alpha \in [0,1],$  and $\Gamma_1$ \color{black} describes the meridian profile of a nodoid
 with  negative constant mean curvature  \cite{Del41,EHLM}.

 From  (\ref{r1d})--(\ref{z1d}),  (\ref{solutions})--(\ref{constraints1}), it follows that
  $\widetilde{\Gamma}_1$ is bounded, respectively, from below and on either side when $\bar{\theta}_1 \le \theta_1 \le 0$, by the lines
 \begin{equation} \label{lines3a}
 z=\max\{ \bar{z}, \, \bar{z}+\tan(\bar{\theta}_3)(r-\bar{r})\}, \quad r=\tilde{r}_1(0), \quad r=\tilde{r}_1(\theta_{\max}), \quad \theta_{\max}=\max\{\bar{\theta}_1, \, -\pi/2\},
 \end{equation}
 and when $0 \le {\theta}_1 \le \theta_c$, by the lines
 \begin{equation} \label{lines3b}
 z=\tilde{z}_1(\theta_c) \ge \tilde{z}_1(\pi), \quad r=\tilde{r}_1(\pi/2), \quad r=\tilde{r}_1(0).
 \end{equation}
 Thus in particular, $\tilde{r}_1(\pi/2) \le \tilde{r}_1(\theta_1) \le \tilde{r}_1(-\pi/2)$  for  $\theta_1 \in [\bar{\theta}_1, \theta_c],$ \color{black}
and   (\ref{solutions})--(\ref{constraints1}), (\ref{constraints3}), (\ref{extG1})--(\ref{r1d}) imply that
  $$0< \tilde{r}_1(\pi/2) < \tilde{r}_1(\pm \pi)= \tilde{r}_1(0)=a_{\ell} <  \tilde{r}_1(\bar{\theta}_1)=\bar{r} \le \tilde{r}_1(\theta_{\max}) \le \tilde{r}_1(-\pi/2), \quad 0<a_{\ell} < \bar{r}< 1.$$
 Since $\Gamma_1 = \widetilde{\Gamma}_1|_{\theta_1\in[\bar{\theta}_1, \theta_c]}$, \color{black} these bounds  hold also for $\Gamma_1,$ which guarantee that the hole persists.

 It follows from (\ref{constraints1}), (\ref{extG1}) that if $-1<\lambda<0,$ then $1 < -1/{\lambda} < \tilde{r}_1(-\pi/2):= - (1 + \sqrt{1 + 4 \lambda^2 a_{\ell}^2})/(2\lambda) $.
 Thus, since $r_1(\theta_{\max})=\tilde{r}_1(-\pi/2)$ when $\bar{\theta}_1 \le -\pi/2$,  in order to guarantee that (\ref{0r1}) is satisfied by (\ref{solutions}) in conjunction with the constraints stated up to now,  we need to impose  the additional constraint
 \begin{equation} \label{constraints4}
 {r}_1(-\pi/2)<1 \quad \hbox{\, if \,} \quad \bar{\theta}_1 < -\pi/2, \quad\quad \rm{(C7)}
 \end{equation}
 in order to guarantee that $\Gamma_1$ does not intersect the inert cylinder. \color{black}
 \par Similarly, it follows from (\ref{z1d}) that
 $$\tilde{z}_1(\pi) \le \tilde{z}_1(\theta_c)=0< \tilde{z}_1(\theta_1) < \tilde{z}_1(0), \quad \theta_1 \in (0, \theta_c),\quad\quad
 \tilde{z}_1(-\pi) < \tilde{z}_1(\bar{\theta}_1)=\bar{z} <\tilde{z}_1(\theta_1) < \tilde{z}_1(0), \quad \theta_1 \in (\bar{\theta}_1,0),$$
  where we have made use of (\ref{angles1}), (\ref{constraints3}), and  the definition of $\widetilde{\Gamma}_1$. \color{black}
 Thus in particular, recalling that $\bar{z}>0$ by (\ref{bar3}), we obtain \color{black} $z_1(\theta_1)=\tilde{z}_1(\theta_1)> 0$ for  $\theta_1 \in [\bar{\theta}_1, \theta_c),$ and $z_1(\theta_c)=\tilde{z}_1(\theta_c)=0;$ hence (\ref{0z}) holds and $\Gamma_1$ intersects the substrate only at $(r_1(\theta_c),z_1(\theta_c))=(r^\ast_1,0)$. \color{black}
 The above implies that $\theta_c= \sup\{ \, {\theta}_1 \, | \,  \tilde{z}_1(\theta_1)>0, \; 0 \le \theta_1 \le \pi \, \}$, and using (\ref{solutions}),(\ref{constraints1}) it follows that
 
 \begin{lemma} \label{z1piSufficient} The requirement
 \begin{equation}\label{z1_theta_c}
z_1(\pi) \le 0, 
\end{equation}
 constitutes a necessary and sufficient condition in order to guarantee that  $z_1(\theta_c)=0,$ as required by (\ref{constraints3}), holds for some $\theta_c \in (0,\, \pi].$
  \end{lemma}
Thus in particular, we obtain
\begin{cor} \label{z1piSection6} The equality $z_1(\pi)=0$ is achieved if and only if $\theta_c=\pi$.  
\end{cor}
The above corollary will be used later in Section \ref{asymptotic analysis}.\color{black}
\bigskip
\par From the geometry of the curve $ \widetilde{\Gamma}_1$, it follows that   the curve $\Gamma_1 \subset \widetilde{\Gamma}_1$ is non-self-intersecting.
We require that $\Gamma_1 \cap \Gamma_2 \cap \Gamma_3=(\bar{r}, \, \bar{z})$, and we have seen above that $\Gamma_2 \cap \Gamma_3=(\bar{r}, \, \bar{z})$.
 So it remains to guarantee that there are no additional intersections of $\Gamma_1$ with $\Gamma_2$ or with $\Gamma_3$.
By looking at the bounds given above on $\Gamma_1$ and $\Gamma_2$, we may conclude that  in fact  $\Gamma_1 \cap \Gamma_2=(\bar{r}, \, \bar{z}).$
 With respect to possible additional intersections of $\Gamma_1$ with $\Gamma_3$, as this appears more difficult to treat
 analytically in complete generality,   we may simply add it to set of constraints that need to be satisfied, in order to conclude that (\ref{solutions}) indeed constitutes a steady state solution, for some given set of parameters, namely
 \begin{equation} \label{constraints5}
 \Gamma_1 \cap \Gamma_3 = (\bar{r},\, \bar{z}).\quad\quad \rm{(C8)}
 \end{equation}
In particular, (\ref{constraints5}) guarantees that  $r_1^* < r_3^*$ as required by (\ref{r0z0}). Note however that if \color{black}
\begin{equation} \label{a_ell_A}
a_{\ell} <A,\quad\quad \rm{(C8')}
\end{equation}
then the bounds on $\Gamma_1$ and $\Gamma_3$ given earlier imply that (\ref{constraints5}) holds. Thus (\ref{a_ell_A}) provides a simple sufficient condition which guarantees that (\ref{constraints5}) holds. 
\bigskip
\par In summary from the discussion above we may conclude
\begin{thm}\label{49} Given $\beta\in (\frac{\pi}{2},\pi)$, $\theta_c\in(0,\pi]$, the expression given in (\ref{solutions}) constitutes a steady state solution to the problem formulated in (\ref{meridian1})--(\ref{l4a}), iff the parameters $\bar{r}$, \color{black} $A,$   $\lambda$,  $a_\ell$ satisfy the additional constraints \rm{(C1)}--(C8) given in (\ref{constraints1}), (\ref{constraints2}), (\ref{constraints3}), (\ref{constraints4}), (\ref{constraints5}). Moreover, if $\rm{(C8')}$ \color{black} given in (\ref{a_ell_A}) holds, then (C8) given in (\ref{constraints5}) is implied. \end{thm}

\subsection{Parameters and Constraints}\label{parameters and constraints}

\par
In Subsection \ref{analytic solutions}, we demonstrated that for given values of the physical parameters, $\beta\in (\frac{\pi}{2},\pi)$ and $\theta_c\in(0,\pi],$ the set of steady states could be prescribed via a set of analytic expressions and constraints, which could be formulated in terms of $4$ parameters $\bar{r}$, $A,$   $\lambda$,  $a_\ell$, see Theorem \ref{49}. Dealing with $4$ parameters in addition to $\beta$ and $\theta_c$ and a variety of constraints is clearly somewhat awkward. Accordingly in this subsection we find it convenient to focus on $2$ parameters, $A$ and $\sigma,$ where $\sigma$ denotes the arclength of $\Gamma_3.$ We then demonstrate that the $4$ parameters listed above can all be expressed in terms of $A,\, \sigma$ and $\beta,$ and afterwards we indicate how the parametric constraints (C1)--(C5),(C7),$\rm{(C8')}$ given in (\ref{constraints1}), (\ref{constraints2}), (\ref{constraints3}), (\ref{constraints4}), (\ref{a_ell_A}) can be formulated in terms of $A,\, \sigma$ and $\beta.$ Thereafter, we consider the constraint (C6) given in (\ref{constraints3}) which we formulate in terms of the parameters $A,\, \sigma,\, \beta$ and $\theta_c$, and then we similarly formulate the constraint (\ref{z1_theta_c}) in terms of the parameters $A,\, \sigma$ and $\beta$. Finally the parametric expressions for (C1)--(C5),(C7),$\rm{(C8')}$ and (\ref{z1_theta_c}) allow us to define an \textit{admissible region}, $\Lambda(A,\sigma,\beta)$, such that each element of $\Lambda(A,\sigma,\beta)$  prescribes a unique steady state solution given by (\ref{solutions}) for some uniquely defined $\theta_c \in (0, \pi]$.\color{black}
\bigskip
\par Let us now consider the parameters $A$ and $\sigma.$ It follows from (\ref{constraints1}) that we need only to consider
\begin{equation}
A\in(0,1).\label{imp1}
\end{equation}
 Recalling that $\bar{r} \in (A,1)$ by (\ref{constraints1}), we obtain
 from (\ref{star}) that
\begin{equation}
\sigma=\int_{r_3(0)}^{r_3(1)}\sqrt{1+\left(\frac{dz_3}{dr_3}\right)^2}dr_3=\sqrt{\bar{r}^2-A^2},\label{sigma}
\end{equation}
with the following constraint on $\sigma$
\begin{equation}
0<\sigma<\sqrt{1-A^2}.\label{imp2}
\end{equation}

\bigskip
Given $A,\ \sigma$ satisfying (\ref{imp1}),(\ref{imp2}),  and given $\beta\in (\frac{\pi}{2},\pi)$, $\theta_c\in(0,\pi]$,
 we now demonstrate how the $3$ parameters $\bar{r}$, $\lambda,$ $a_\ell,$ as well as the angles $\bar{\theta}_1,$ $\bar{\theta}_2,$ $\bar{\theta}_3,$ and $\bar{z}$, \color{black} can  be expressed in terms of $A$, $\sigma$,  $\beta$, $\theta_c$.
\bigskip\par\noindent
Let us begin with the parameters pertaining to the triple junction, namely: $\bar{r},$ $\bar{z},$ $\bar{\theta}_1,$ $\bar{\theta}_2,$ $\bar{\theta}_3$.
 Using (\ref{constraints1}), (\ref{bar3}), and (\ref{sigma}), we obtain for the triple junction coordinates $\bar{r},\ \bar{z}$ that
  \begin{equation}
\bar{r}=\sqrt{\sigma^2+A^2},\quad\quad \bar{z}=A\,\acosh\left(\frac{\sqrt{\sigma^2+A^2}}{A}\right)\label{rzbar}.
\end{equation}
With regard to the angles at the triple junction $\bar{\theta}_1,$ $\bar{\theta}_2,$ $\bar{\theta}_3$, we get
from (\ref{bar3}), (\ref{rzbar}),  that
\begin{equation}
\bar{\theta}_3=\arcsin\left(\frac{A}{\sqrt{\sigma^2+A^2}}\right),\quad \bar{\theta}_3\in\left(0,\frac{\pi}{2}\right).\label{psi3}
\end{equation}
From  (\ref{bar2}), (\ref{psi3}), we now obtain
\begin{equation}
\bar{\theta}_2=\bar{\theta}_3+\beta-\pi=\arcsin\left(\frac{A}{\sqrt{\sigma^2+A^2}}\right)+\beta-\pi,\quad \sin(\bar{\theta}_2)=\frac{-A\,\cos(\beta)-\sigma\sin(\beta)}{\sqrt{\sigma^2+A^2}},\quad 0<\bar{\theta}_2<\bar{\theta}_3.\label{psi2}
\end{equation}
From (\ref{angles1}), (\ref{psi3}), we get that
\begin{equation}
\bar{\theta}_1=\bar{\theta}_3-\beta=\arcsin\left(\frac{A}{\sqrt{\sigma^2+A^2}}\right)-\beta,\quad \sin(\bar{\theta}_1)=\frac{A\,\cos(\beta)-\sigma\sin(\beta)}{\sqrt{\sigma ^2+A^2}},\quad \bar{\theta}_1\in(-\pi,0).\label{psi1}
\end{equation}

\bigskip\par\noindent

With regard to $\lambda$ and $a_l,$ we first obtain an expression for $\lambda.$ Recalling  (\ref{constraints2}), (\ref{rzbar}) \color{black}we get that
 \begin{equation}
\sqrt{A^2+\sigma^2}=\frac{\sin(\bar{\theta}_2)-\sqrt{\sin^2(\bar{\theta}_2)+4\lambda^2}}{2\lambda},\nonumber
\end{equation}
then using (\ref{constraints1}) and (\ref{psi2}), 
\begin{equation}
\lambda=\frac{A\,\cos(\beta)+\sigma\sin(\beta)}{1-A^2-\sigma^2}.\label{lambda}
\end{equation}
 \color{black} Finally to obtain an expression for $a_\ell,$  we recall (\ref{solutions}) and (\ref{constraints3}). Then (\ref{rzbar}) \color{black} implies that
\begin{equation}
\sqrt{A^2+\sigma^2}=\frac{\sin(\bar{\theta}_1)-\sqrt{\sin^2(\bar{\theta}_1)+4\lambda^2a_\ell^2}}{2\lambda},\nonumber
\end{equation}
and using (\ref{constraints1}), (\ref{psi1}), and (\ref{lambda})  we obtain that 
\begin{equation}
a_\ell=\sqrt{\frac{(2A^2+2\sigma^2-1)A\,\cos(\beta)+\sigma\sin(\beta)}{A\,\cos(\beta)+\sigma\sin(\beta)}}.\label{al}
\end{equation}

\bigskip
\par Next, we formulate the constraints (C1)--(C5),(C7) given in (\ref{constraints1}), (\ref{constraints2}), (\ref{constraints3}), (\ref{constraints4}), as well as the sufficient constraint $\rm{(C8')}$ given in (\ref{a_ell_A}), in terms of $A,\ \sigma,$ $\beta$. Afterwards we return to consider the constraint (C6) given in (\ref{constraints3}) which depends on $A,\ \sigma,$ $\beta,\ \theta_c.$\color{black}
\bigskip 
\par We start by considering (C1)--(C3) given in (\ref{constraints1}). From (\ref{imp1}), (\ref{imp2}), (\ref{rzbar}),
\begin{equation}
0<A<\bar{r}=\sqrt{\sigma^2+A^2}<1.\label{constraint1a}
\end{equation}
From (\ref{lambda}), (\ref{constraint1a}) the constraint $\lambda<0,$ can be formulated as
$$A\cos(\beta)+\sigma\sin(\beta)<0.$$ Since $A^2+\sigma^2<1$ by (\ref{constraint1a}), using (\ref{al}), the constraint $0<a_l<1$ can be formulated as $$(2A^2+2\sigma^2-1)A\,\cos(\beta)+\sigma\sin(\beta)<0,$$ or equivalently,
\begin{equation}
\tan(\beta)>\frac{A}{\sigma}(1-2A^2-2\sigma^2).\label{imp4}
\end{equation}
Note that (\ref{constraint1a})--(\ref{imp4}) imply in particular that $\tan(\beta)>-\frac{A}{\sigma}$. Hence, the constraint $A\,\cos(\beta)+\sigma\sin(\beta)<0$ is redundant. 
\bigskip
\par Let us consider (C4) given in (\ref{constraints2}). Under the assumption that (C1)--(C3) are satisfied, and hence that (\ref{constraint1a})--(\ref{imp4}) hold, (C4) implies the parametric representation of $\lambda$ given in (\ref{lambda}). Note that if (\ref{lambda}),(\ref{constraint1a}),(\ref{imp4}) are satisfied, then the constraint (C4) is also satisfied.
\par Similarly, the constraint (C5) in (\ref{constraints3}) is already satisfied, since (\ref{al})--(\ref{imp4}) imply that $r_1(\bar{\theta}_1)=\bar{r}$.
\bigskip 
\par Next, let us consider (C7) given in (\ref{constraints4}). Using (\ref{psi1}) (see \cite{MScKG} for details), the inequality $\bar{\theta}_1<-\frac{\pi}{2}$ can be formulated as
$$\cos(\beta)<-\frac{A}{\sqrt{\sigma^2+A^2}},$$
then using (\ref{solutions}),(\ref{lambda})--(\ref{imp4}), the constraint $r_1(\theta_1=-\frac{\pi}{2})<1$ can be formulated \cite{MScKG} as $$\cos(\beta)<-\frac{1}{2A}.$$ Thus (C7) given in (\ref{constraints4}) can be expressed as
\begin{equation}
\cos(\beta)<-\frac{1}{2A}\quad \text{if}\quad \cos(\beta)<-\frac{A}{\sqrt{\sigma^2+A^2}}.\label{imp5}
\end{equation}
\\While the constraint (C8) given in (\ref{constraints5})
is not straightforward to formulate as a parametric constraint, using (\ref{al})--(\ref{imp4}), the constraint  $\rm{(C8')}$ given in (\ref{a_ell_A}) may be reformulated \cite{MScKG} as\color{black}
\begin{equation}
\tan(\beta)<\frac{A(1-A^2-2\sigma^2)}{\sigma(1-A^2)}.\label{imp6}
\end{equation}

\smallskip
Finally we consider the constraint (C6) given in (\ref{constraints3}), namely $z_1(\theta_c)=0.$ 
Let us note that by (\ref{solutions}), it can be expressed as
\begin{equation}\label{value_thetac}
z_1(\theta_c)= \bar{z} - \frac{1}{2\lambda} \int_{\bar{\theta}_1}^{\theta_c}{\Biggl[\frac{\sin^2(x)}{\sqrt{\sin^2(x)+4\lambda^2a_\ell^2}}-\sin(x)\Biggr]dx}=0.
\end{equation}   
Since $\bar{z},\bar{\theta}_1$, and the parameters $\lambda,a_\ell$ have all been expressed above in terms of $(A,\sigma,\beta)$, (\ref{value_thetac}) implies a parametric constraint on $(A,\sigma,\beta,\theta_c)$. Recalling Lemma \ref{z1piSufficient}, we obtain that (\ref{value_thetac}) holds for some $\theta_c\in(0,\pi]$ if and only if (\ref{z1_theta_c}) is satisfied, namely $z_1(\pi)\le 0.$ Using (\ref{solutions}), since  $\bar{z},\bar{\theta}_1,\lambda,a_\ell$  have been expressed above in terms of $(A,\sigma,\beta)$, (\ref{z1_theta_c}) implies a parametric constraint on $(A,\sigma,\beta)$.\color{black}
\bigskip
\par Collecting the constraints (C1)--(C5),(C7), $\rm{(C8')}$ formulated in terms of $A,\,\sigma$ and $\beta$ as given in (\ref{constraint1a})--(\ref{imp6}), together with the constraint (\ref{z1_theta_c}), we obtain a set of sufficient conditions for a given $(A, \sigma, \beta)\in \mathbb{R}^2\times(\frac{\pi}{2},\pi)$ to prescribe a unique steady state solution given in (\ref{solutions}) to the problem formulated in (\ref{meridian1})--(\ref{l4a}). In fact, all the conditions above but  $\rm{(C8')}$ are also necessary. Thus, in order to satisfy the conditions listed above we first restrict our set of parameters $(A,\sigma,\beta)$ to the domain $(A, \sigma,\beta)\in \Omega_0\times(\frac{\pi}{2},\pi),$ where $\Omega_0:=\{ (A, \sigma)\,|\, 0<A<1, \, 0<\sigma<\sqrt{1-A^2}\},$ which will be used in Section \ref{asymptotic analysis}. Then the resulting set of parameters is further restricted to the \textit{admissible $A,\sigma,\beta$  region}, $\Lambda=\Lambda(A, \sigma,\beta)$, which is defined as the set of all the parameters $(A, \sigma,\beta)\in \Omega_0\times(\frac{\pi}{2},\pi)$ satisfying the following additional constraints:\color{black}
\begin{equation}\label{constraints}
\begin{cases}
&\tan(\beta)>\frac{A}{\sigma}(1-2A^2-2\sigma^2),\\[1ex]
&\cos(\beta)<-\frac{1}{2A} \quad\hbox{if}\quad \cos(\beta)<-\frac{A}{\sqrt{\sigma^2+A^2}},\\[1ex]
&z_1(\pi)\le 0,\\[1ex]
&\tan(\beta)<\frac{A(1-A^2-2\sigma^2)}{\sigma(1-A^2)}.
\end{cases}
\end{equation}	
Hence, given $(A,\sigma,\beta) \in \Lambda(A, \sigma,\beta),$ there is a unique correspondence with a steady state solution, (\ref{solutions}), where $\theta_c \in (0,\pi]$ is uniquely determined by (\ref{value_thetac}) in accordance with (\ref{z1d}) and Lemma \ref{z1piSufficient},  and where the parameters  in (\ref{solutions}), (\ref{value_thetac}) are given in (\ref{rzbar})--(\ref{al}).
\color{black}
\smallskip
\par The  results above imply the following
\begin{thm} \label{oneone}
	Given  $(\beta,\theta_c) \in (\frac{\pi}{2}, \pi)\times (0, \pi]$, each set of parameters  $(A,\sigma,\beta,\theta_c)$  satisfying  $(A,\sigma,\beta) \in \Lambda(A, \sigma,\beta)$ and (\ref{value_thetac}), prescribes a subset of all possible steady state solutions to (\ref{meridian1})--(\ref{l4a}).
\end{thm} \color{black}

\begin{thm} \label{onetwo}
	Given  $\beta \in (\frac{\pi}{2}, \pi)$, each set of parameters  $(A,\sigma,\beta) \in \Lambda(A, \sigma,\beta)$ prescribes a subset of all possible steady state solutions to (\ref{meridian1})--(\ref{l4a}) for some  values of  $\theta_c\in(0,\pi]$.
\end{thm} 
Below we list the functional dependencies for the parameters needed in defining a solution, given $(A,\sigma,\beta)\in\Lambda(A,\sigma,\beta),$

\begin{equation}\label{parameters}
\begin{cases}
&\bar{r}=\sqrt{\sigma^2+A^2},\quad 0<\bar{r}<1,\quad \quad \bar{z}=A\,\acosh\left(\frac{\sqrt{\sigma^2+A^2}}{A}\right),\quad \bar{z}>0,\\\\
&\bar{\theta}_3=\arcsin\left(\frac{A}{\sqrt{\sigma^2+A^2}}\right),\quad \bar{\theta}_1=\bar{\theta}_3-\beta,\quad \bar{\theta}_2= \bar{\theta}_3 +\beta-\pi,\quad\quad 0<\bar{\theta}_3<\frac{\pi}{2},\; -\pi<\bar{\theta}_1<0,\;0<\bar{\theta}_2<\bar{\theta}_3,\\\\
&\lambda=\frac{A\,\cos(\beta)+\sigma\,\sin(\beta)}{1-A^2-\sigma^2},\quad \lambda<0,\quad\quad a_\ell=\sqrt{\frac{(2A^2+2\sigma^2-1)A\,\cos(\beta)+\sigma\,\sin(\beta)}{A\,\cos(\beta)+\sigma\,\sin(\beta)}},\quad 0<a_\ell<1.
\end{cases}
\end{equation}
In the above,  the range conventions: $\acosh: [1,\infty) \rightarrow [0,\infty)$, $\arcsin: [-1,1]\rightarrow [-\pi/2, \pi/2],$ are being assumed. Accordingly, the functions $\bar{z}$, $\bar{\theta}_1,$ $\bar{\theta}_2,$ $\bar{\theta}_3,$  listed in (\ref{parameters}), are all well-defined and belong to the ranges indicated for $(A,\sigma,\beta) \in \Lambda(A,\sigma,\beta).$\color{black}\smallskip
\par Finally we remark that it is easy to show that the necessary and sufficient conditions stated in Theorem \ref{49} may be reformulated in a more parametric manner as follows
\begin{thm} \label{onethree}
	Given  $(\beta,\theta_c) \in (\frac{\pi}{2}, \pi)\times (0, \pi]$, each set of parameters  $(A,\sigma)\in\Omega_0$ prescribes a steady state solution to (\ref{meridian1})--(\ref{l4a}) given by (\ref{solutions}) for $(\beta,\theta_c)$ iff (\ref{parameters}) holds and
	\begin{equation}\label{constraintsnessuf}
	\begin{cases}
	&\tan(\beta)>\frac{A}{\sigma}(1-2A^2-2\sigma^2),\\[1ex]
	&\cos(\beta)<-\frac{1}{2A} \quad\hbox{if}\quad \cos(\beta)<-\frac{A}{\sqrt{\sigma^2+A^2}},\\[1ex]
	&z_1(\theta_c)=0,\\[1ex]
	&\Gamma_1 \cap \Gamma_3 = (\bar{r},\, \bar{z}).
	\end{cases}
	\end{equation}	
	
\end{thm} \color{black}

\subsection{The Energy and Volume at Steady State}\label{energy volume_ss}
\hfill
\par It \color{black}follows from Lemma \ref{dedt} that $\frac{d}{dt}\mathbb{E} \le 0$  for $t > 0$, and $\frac{d}{dt}\mathbb{E} = 0$ \color{black} at steady state. Using (\ref{etot}), (\ref{aex}) and introducing the variables $\theta_i,\ i=1,2,3$,   the total free energy at steady state
may be expressed as
\begin{equation}
\mathbb{E}=\pi\gamma_{grs}+2\pi \gamma_{ex}\left(\int_{\bar{\theta}_1}^{\theta_c}{r_1\sqrt{r_{1\theta_1}^2+z_{1\theta_1}^2}d\theta_1}+\int_{0}^{\bar{\theta}_2}{r_2\sqrt{r_{2\theta_2}^2+z_{2\theta_2}^2}d\theta_2}\right)+
\pi\gamma_{ex} r_1^{2}(\theta_c)\cos(\theta_c)+ 2 \pi\gamma_{gb}\int_{\bar{\theta}_3}^{\frac{\pi}{2}}{r_3\sqrt{r_{3\theta_3}^2+z_{3\theta_3}^2}d\theta_3}. \label{ee}
\end{equation}
For    $\mathbb{E}_{eff}=(\mathbb{E}-\pi\gamma_{grs})/({\pi\gamma_{ex}})$, the effective energy,    we obtain that  at steady state,
\begin{equation}
\mathbb{E}_{eff}=2\left(\int_{\bar{\theta}_1}^{\theta_c}{r_1\sqrt{r_{1\theta_1}^2+z_{1\theta_1}^2}d\theta_1}+\int_{0}^{\bar{\theta}_2}{r_2\sqrt{r_{2\theta_2}^2+z_{2\theta_2}^2}d\theta_2}\right)+
\cos(\theta_c)r_1^{2}(\theta_c)+2m\int_{\bar{\theta}_3}^{\frac{\pi}{2}}{r_3\sqrt{r_{3\theta_3}^2+z_{3\theta_3}^2}d\theta_3},\label{eeff}
\end{equation}
where $m:={\gamma_{gb}}/{\gamma_{ex}}.$
It also follows easily from Lemma \ref{dedt} that $\frac{d}{dt}\mathbb{E}_{eff} \le 0$  for $t > 0$, and $\frac{d}{dt}\mathbb{E}_{eff} = 0$  at steady state.
\par According to Lemma \ref{dvdt}, $\frac{d}{dt}\mathbb{V} = 0$ and hence $\mathbb{V}(t)=\mathbb{V}(0)$ for $t\ge 0$. Therefore, the volume  at steady state is determined by the volume of the initial conditions. Using (\ref{volume}) and the variables $\theta_i,\ i=1,2$, we obtain that at steady state,
\begin{equation}
\mathbb{V}=2\pi\left(\int_{\theta_c}^{\bar{\theta}_1}\ z_1r_1r_{1\theta_1}d\theta_1+\int_{\bar{\theta}_2}^{0}\ z_2r_2r_{2\theta_2}d\theta_2\right).\label{volume2}
\end{equation}
\par Note that the energy decay and volume conservation imply additional constraints to the list of necessary and sufficient conditions, given by (\ref{constraints}), for steady state solutions to (\ref{meridian1})--(\ref{l4a}) arising from prescribed initial conditions, see (\ref{ic}). However, we shall not undertake a complete study of the implications of the energy and volume constraints in the present study.

\section{Numerical Steady State Solutions}\label{numerical solutions}

\vspace{1mm}
In Subsection \ref{analytic solutions}, a set of parametric necessary and sufficient conditions, $\rm{(C1)-(C8)}$,  were derived, which prescribe the parameters used in (\ref{solutions}) corresponding to steady  state solutions to (\ref{meridian1})--(\ref{l4a}) for given values of  $\beta\in\left(\frac{\pi}{2},\pi\right),\theta_c\in(0,\pi]$ which are determined by some initial conditions (\ref{ic}).  Afterwards, a somewhat simpler to formulate set of sufficient parametric conditions, $\rm{(C1)-(C7), (C8')}$ were derived, which could be expressed  analytically via condition (\ref{value_thetac}) and the set of four conditions in (\ref{constraints}) to be satisfied by  $(A, \sigma)\in \Omega_0$. Using MATLAB, we calculated $A-\sigma$ curves corresponding to steady states for various values of $(\beta,\theta_c)$ based on the sufficient conditions, although this may imply that we overlooked some steady states. See Fig. \ref{SigmaAcombined}.
We then focused on the case of the physically realistic values $\beta=1.72,\ \theta_c=11\pi/18,$ and identified an $A-\sigma$ curve corresponding to those values based on the sufficient conditions, which is portrayed in Fig. \ref{VolumeSigmaA}a. For the three points indicated on the $A-\sigma$ curve in Fig. \ref{VolumeSigmaA}a, the corresponding steady states and their volumes were calculated based on the formulas in \S\ref{stationary}, see Fig. \ref{VolumeSigmaA}b. For details of the numerical procedures, see \cite{MScKG}.
\begin{figure}[H]
	\centering
	\includegraphics[width=.9\textwidth]{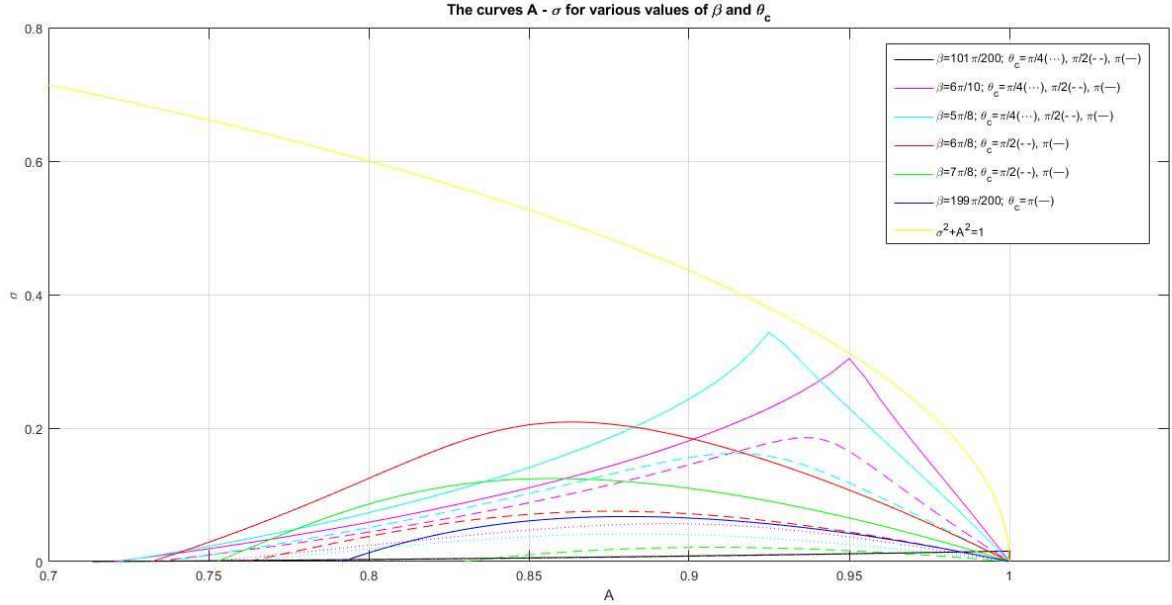}
	\caption{The $A-\sigma$ curves which correspond to the steady state solutions which were obtained for various values of $\beta$ and $\theta_c$. These curves appear to be  continuous and
		describable as a graph of a function $\sigma=\sigma(A; \beta, \theta_c)$ with $\sigma(A;\beta,\theta_c)\rightarrow 0$ as $A\rightarrow 1^- .$  For a given value of $\beta\in(\frac{\pi}{2},\pi),$ the $A-\sigma$ curves corresponding to different $\theta_c$ values do not intersect;  more specifically, the curves corresponding to smaller values of $\theta_c$ are enclosed by the curves corresponding to larger values of $\theta_c.$ Moreover, the $A-\sigma$ curves appear to become more singular as  $\beta$ decreases or as  $\theta_c$ increases.}\label{SigmaAcombined}
\end{figure}
\noindent

\begin{figure}[H]
	\centering
	\begin{minipage}[h]{0.44\textwidth}
		\centering
		\textbf{a}
		\includegraphics[width=.9\textwidth]{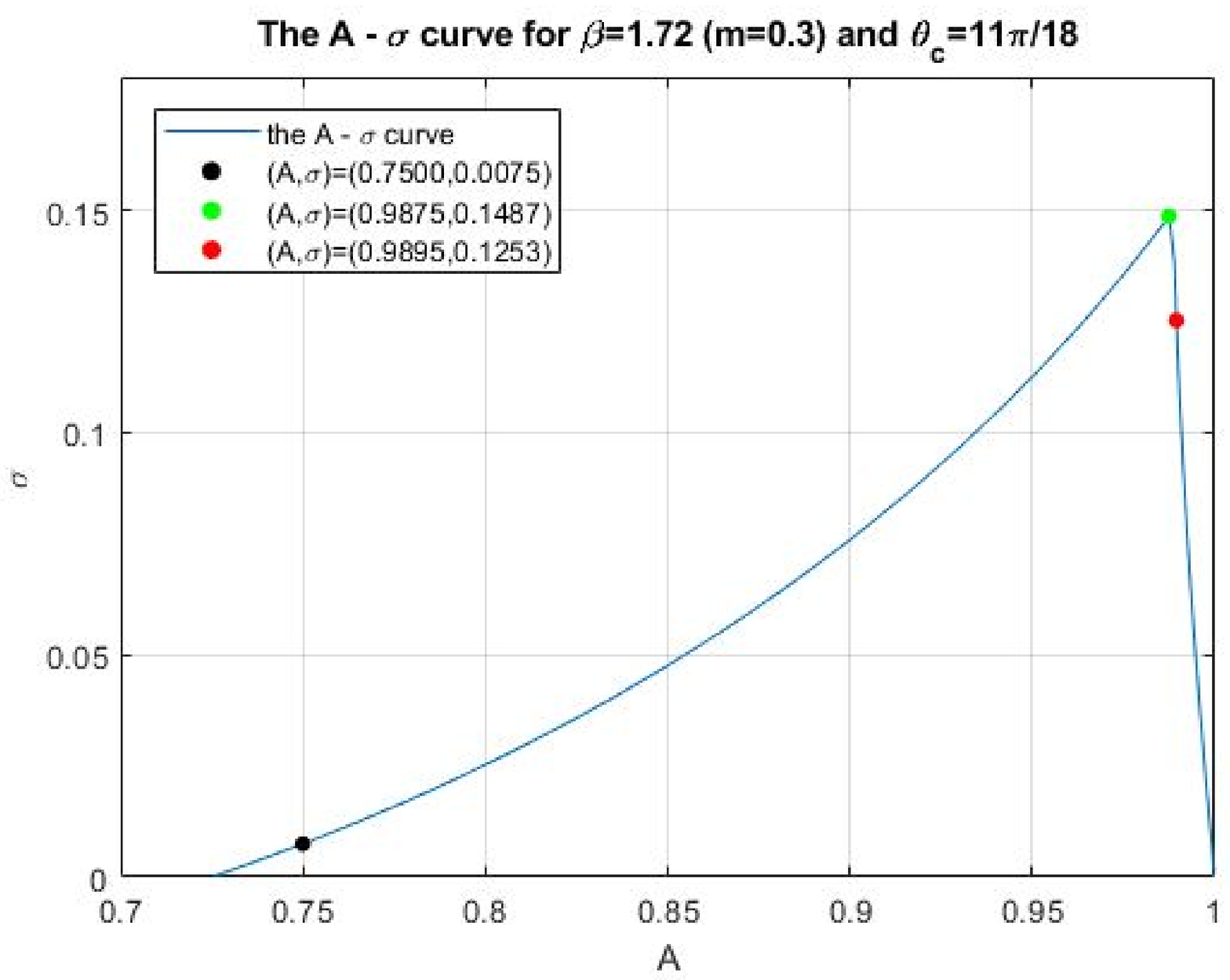}
	\end{minipage}
	\hfill
	\begin{minipage}[h]{0.44\textwidth}
		\centering
		\textbf{b}
		\includegraphics[width=.9\textwidth]{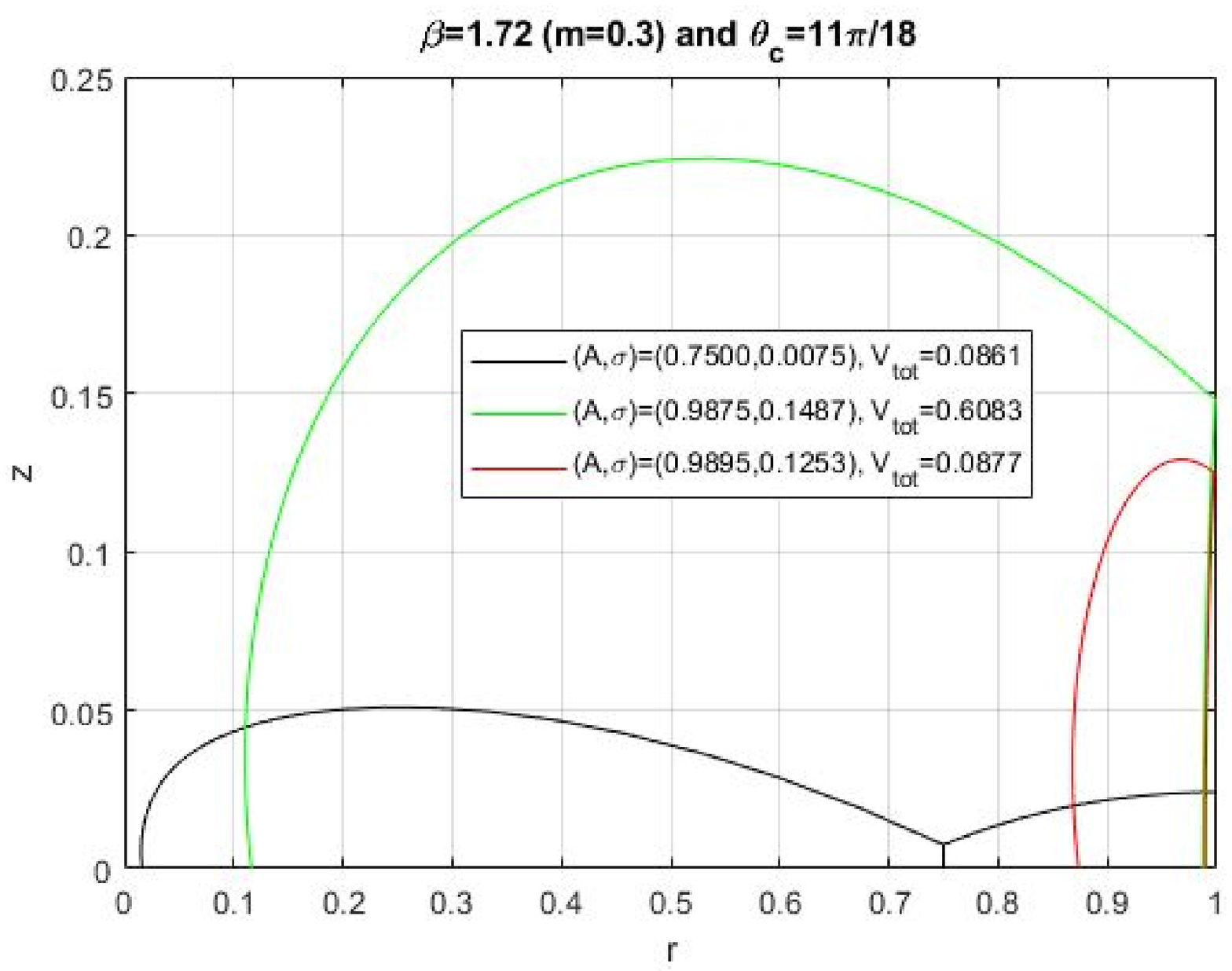}
	\end{minipage}
	\caption{Here $\beta=1.72$ ($m=0.3$) and $\theta_c=\frac{11\pi}{18}$, which are physically realistic values for metallic films on ceramic substrates. a) The $A-\sigma$ curve corresponding to the steady states.
			b) A combined plot of the calculated profiles where the black, green and red profile colors indicate that their respective $(A, \sigma)$ values are in accordance with the color coding used in a), and $\rm{V_{tot}}$ is the total volume $\mathbb{V}$ enclosed by the corresponding steady state solution.}\label{VolumeSigmaA}
\end{figure}

\newpage
\bigskip

\color{black}\section{Asymptotic analysis}\label{asymptotic analysis}

In looking at Fig. \ref{SigmaAcombined},  it can be seen that
for all values of $(\beta, \theta_c)$   considered, the set of possible steady states can be parametrized  as  curves $(A, \sigma)$ in $\Omega_0$, which
depend on $(\beta, \theta_c).$  More specifically, these parametric curves all appear to be describable as graphs of functions $\sigma=\sigma(A; \beta, \theta_c)$ and to satisfy
$\lim_{A \rightarrow 1^-} \sigma(A; \beta, \theta_c)=0$. In this section we provide an analytical proof of this result for the case of $\theta_c=\pi$ in a neighborhood of $(A,\sigma)=(1,0)$;  namely,
 \begin{thm} \label{implicit}
	For any  $\beta \in (\frac{\pi}{2}, \pi) $ and for $\theta_c=\pi$,  the set of steady state solutions to (\ref{meridian1})--(\ref{l4a}) can be parametrized in $\Omega_0$ by a continuous function $\sigma=\sigma(A; \beta, \theta_c=\pi)$ in the
	neighborhood of the point $(A,\sigma)=(1,0),$ where $(A, \sigma(A; \beta,\theta_c=\pi), \beta)$  satisfies (\ref{constraints}).  Moreover,  $\lim_{A \rightarrow 1^-} \sigma(A;\beta,\theta_c=\pi)=0$.
\end{thm}

\begin{proof}
	Let us now focus on  a small neighborhood of $(A,\sigma)=(1,0)$ within $\Omega_0$. More specifically,
		given $\beta\in \left(\frac{\pi}{2},\pi\right)$ and $0<\tilde{\epsilon}\ll 1$, let us define the following neighborhood of $(A,\sigma)=(1,0)$,
		\begin{equation}
		\Omega_{\tilde{\epsilon}}:=\{(A,\sigma) \,| \,  1-\tilde{\epsilon}<A< 1,\, 0<\sigma<\sqrt{1-A^2} \},\label{omega}
		\end{equation}
		where by definition, $\Omega_{\tilde{\epsilon}} \subset \Omega_0$.
According to Theorem \ref{oneone} if $(A,\sigma) \in \Omega_{\tilde{\epsilon}}$, and if the constraints in (\ref{constraints}) are satisfied for $(A,\sigma;\beta,\theta_c=\pi)$, with the third inequality is satisfied as an equality, then $(A,\sigma;\beta,\theta_c=\pi)$ uniquely prescribe a steady state solution to (\ref{meridian1})--(\ref{l4a}). Moreover, by considering the specific geometry of the steady state solutions in the special case of $\theta_c=\pi$, the necessary and sufficient condition $\rm{(C8)}$ is readily seen to imply $\rm{(C8')}$. In other words, the fourth constraint in (\ref{constraints}), which is a sufficient condition for satisfying (\ref{constraints5}), is also necessary. Thus in $\Omega_{\tilde{\epsilon}}$, the constraints in (\ref{constraints}) for $(A,\sigma;\beta,\theta_c=\pi)$ with the third inequality satisfied as an equality are both necessary and sufficient. Accordingly, given $\beta\in\left(\frac{\pi}{2},\pi\right)$, then  $(A,\sigma;\beta,\theta_c=\pi)$ prescribe a steady state solution for $(A,\sigma) \in \Omega_{\tilde{\epsilon}}$ if and only if the following necessary and sufficient conditions are satisfied
	\begin{eqnarray}
	{\rm{(I)}}\ \tan(\beta)>\frac{A}{\sigma}(1-2A^2-2\sigma^2),&\quad\quad&{\rm{(II)}}\ \cos(\beta)<-\frac{1}{2A} \quad\hbox{if}\quad \cos(\beta)<-\frac{A}{\sqrt{\sigma^2+A^2}},\nonumber\\
	{\rm{(III)}}\  z_1(\pi)=0,&\quad\quad&{\rm{(IV)}}\ \tan(\beta)<\frac{A(1-A^2-2\sigma^2)}{\sigma(1-A^2)}.\nonumber
	\end{eqnarray}

\begin{claim}\label{claim1} If  $\tilde{\epsilon}>0$ is sufficiently small, then the conditions $\rm{(I), (II)}$ are satisfied within $\Omega_{\tilde{\epsilon}}$.
\end{claim}
\begin{proof}[Proof of Claim \ref{claim1}]
Clearly condition (II) is satisfied if
\begin{equation}
2A^2>\sqrt{\sigma^2+A^2}.\label{22}
\end{equation}
Since $\sqrt{\sigma^2+A^2}<1$,
we obtain that if $2(1-\tilde{\epsilon})^2 >1$, then (\ref{22}) is satisfied within $\Omega_{\tilde{\epsilon}}$. Therefore, we require
\begin{equation}
\tilde{\epsilon}<1-{1}/{\sqrt{2}}.\label{222}
\end{equation}
Furthermore, since  $0<\sigma<\sqrt{2\tilde{\epsilon}} \ll 1$ for $(A,\sigma)\in\Omega_{\tilde{\epsilon}}$, we obtain
$$\frac{A}{\sigma}(2A^2+2\sigma^2-1)>\frac{1-\tilde{\epsilon}}{\sqrt{2\tilde{\epsilon}}}(2(1-\tilde{\epsilon})^2-1)>\frac{1-8\tilde{\epsilon}}{2\sqrt{\tilde{\epsilon}}}.$$
Hence condition (I) is satisfied if $0<\tilde{\epsilon} < 1- 1/\sqrt{2}$ is taken sufficiently small so that
\begin{equation} 
-\frac{1-8\tilde{\epsilon}}{2\sqrt{\tilde{\epsilon}}} < \tan(\beta).\label{11}
\end{equation}
\end{proof}
\bigskip
Note that the upper bound on $\tilde{\epsilon}$  in (\ref{11}) depends on our choice of $\beta\in (\pi/2,\pi)$, and in particular $\tilde{\epsilon}<\frac{1}{8}.$
  
\begin{claim}\label{claim2} If  $\tilde{\epsilon}>0$ is sufficiently small, then condition $\rm{(III)}$ holds within $\Omega_{\tilde{\epsilon}}$ if and only if $(A,\sigma)$ belong to a locus of points describable as follows
\begin{equation}\label{SolutionCurve}
\sigma(A;\beta,\theta_c=\pi)=\Bigl[\frac{\sin(\beta)+1}{-\cos(\beta)}\Bigr](1-A) + O( (1-A)^2).
\end{equation}

\end{claim}
\begin{proof}[Proof of Claim \ref{claim2}]
Let us now set
\begin{equation}\label{defGG}
G(A,\sigma;\beta,\theta_c=\pi):=-2\lambda kz_1(\pi),
\end{equation}
where $k$, defined in (\ref{kk'}) in the Appendix, is positive. From the definition of $\lambda$ in (\ref{parameters}) and condition (I), it follows that $\lambda k< 0$ for $(A,\sigma)\in \Omega_{\tilde{\epsilon}}$. Thus, condition (III) is satisfied if and only if $G$ vanishes. 

We proceed now to analyze the locus where $G(A,\sigma;\beta,\theta_c=\pi)=0$
for $(A,\sigma)\in\Omega_{\tilde{\epsilon}}$ with $\tilde{\epsilon}>0$  sufficiently small. From (\ref{defGG}), and using the expression derived for $z_1(\pi)$ given in (\ref{Gpi}) in the Appendix, we obtain
	\begin{equation}
	G(A,\sigma;\beta,\theta_c=\pi)=k'^2\left(F\left(\bar{\theta}_1+\frac{\pi}{2},k\right)-3K(k)\right)-\left(E\left(\bar{\theta}_1+\frac{\pi}{2},k\right)-3E(k)\right)-2\lambda k \bar{z}-k\cos(\bar{\theta}_1)-k.\label{implicitG}
	\end{equation}
	Using (\ref{parameters}), (\ref{incomplete}), (\ref{complete}), (\ref{k}), (\ref{kk}) we notice that for $(A,\sigma)\in \Omega_{\tilde{\epsilon}}$ the terms:$$\bar{z},\ \bar{\theta}_1,\ \lambda,\ k,\ k',\ F\left(\bar{\theta}_1+\frac{\pi}{2},k\right),\ E\left(\bar{\theta}_1+\frac{\pi}{2},k\right),\ K(k),\ E(k)$$
	are continuously differentiable functions of $A,\ \sigma$ and hence $G$ is also continuously differentiable function of $A,\ \sigma$ in this neighborhood.\\

\bigskip
We first evaluate $G$ along the lower limit of $\sigma$ within $\Omega_{\tilde{\epsilon}}$, namely when $\sigma \downarrow 0$ for $(A,\sigma)\in \Omega_{\tilde{\epsilon}}$.
It follows from (\ref{parameters}),(\ref{k}) that in this limit,
$$\bar{\theta}_1+ \frac{\pi}{2}=\pi-\beta,\quad \bar{z}=0,\quad k=\frac{1-A^2}{\sqrt{(1-A^2)^2 - 8A^2 \cos^2(\beta)(1-A^2)+ 4 A^2 \cos^2(\beta)}},$$
$$\lambda k=\frac{A \cos(\beta)}{\sqrt{(1-A^2)^2 - 8A^2 \cos^2(\beta)(1-A^2)+ 4 A^2 \cos^2(\beta)}}.$$
Since $A= 1 + O(\tilde{\epsilon})$, $0<1-A^2<2\tilde{\epsilon} + O(\tilde{\epsilon}^2),$ and by the definition of $k'$ in (\ref{kk'}) we obtain 
$$0< k < \frac{\tilde{\epsilon}}{-\cos(\beta)}+ O({\tilde{\epsilon}}^2), \quad \lambda k = -\frac{1}{2}+ O(\tilde{\epsilon}),\quad  k'^2=1 +O({\tilde{\epsilon}}^2).$$
Substituting the above estimates into the expression for $G(A,\sigma;\beta,\theta_c=\pi)$ given in (\ref{implicitG}), and using  (\ref{incomplete2})--(\ref{complete2}), it is easy to check that
$G=  k(-1- \sin(\beta)) + O(k^2)<0$ when $\tilde{\epsilon}$ is sufficiently small, along the lower limit of $\sigma$ within $\Omega_{\tilde{\epsilon}}$.

\bigskip
	We now continue by evaluating $G$ along the upper limit of $\sigma$ within $\Omega_{\tilde{\epsilon}}$, namely when $\sigma\uparrow \sqrt{1-A^2}$ for $(A,\sigma)\in \Omega_{\tilde{\epsilon}}$. From (\ref{parameters}),(\ref{k}),(\ref{kk}) we obtain in this limit
	\begin{align}\label{Glow1}
	\begin{array}{l}
		\bar{\theta}_1+\frac{\pi}{2} \rightarrow \arcsin(A)-\beta+\frac{\pi}{2},\quad \bar{z} \rightarrow A\,\acosh\left(\frac{1}{A}\right), \quad\\[1.5ex]
 k \rightarrow 0, \quad
	k'^2 \rightarrow 1, \quad
	\lambda k \rightarrow -\frac{1}{2},\quad
		k\cos(\bar{\theta}_1) \rightarrow 0.
	\end{array}
	\end{align}
	From (\ref{incomplete2}), (\ref{complete2}) we obtain that
	\begin{equation}\label{Glow2}
	\underset{k\rightarrow 0}{\lim}K(k)=	\underset{k\rightarrow 0}{\lim}E(k)=\frac{\pi}{2},\quad\underset{k\rightarrow 0}{\lim}F(\bar{\theta}_1+\frac{\pi}{2},k)=	\underset{k\rightarrow 0}{\lim}E(\bar{\theta}_1+\frac{\pi}{2},k)=\bar{\theta}_1+\frac{\pi}{2}.
	\end{equation}
	Hence by (\ref{implicitG}), (\ref{Glow1}), (\ref{Glow2}),
	\begin{equation}\label{positiveG}
	\underset{\sigma\rightarrow \sqrt{1-A^2}}{\lim}G=A\,\acosh\left(\frac{1}{A}\right)>0.
	\end{equation}

From the discussion above it follows that within $\Omega_{\tilde{\epsilon}}$, for any $A$ there exists $\sigma$ such that
$G(A,\sigma;\beta,\theta_c=\pi)=0$. However,
in order to show that such $\sigma$ is unique, namely $\sigma=\sigma(A;\beta,\theta_c=\pi)$ is describable as a function of $A$,  
we must refine our argument.  If we can show that
\begin{equation} \label{Gsigma}
G_{\sigma}(A,\sigma;\beta,\theta_c=\pi)>0,\quad (A,\sigma)\in \Omega_{\tilde{\epsilon}},
\end{equation}
 then uniqueness is guaranteed. Since, as noted above, $G$ is a continuously differentiable
 function of $A,\ \sigma$ within $\Omega_{\tilde{\epsilon}}$ and there exists $(A,\sigma)\in \Omega_{\tilde{\epsilon}}$ such that
 $G(A,\sigma;\beta,\theta_c=\pi)=0$, it follows that if (\ref{Gsigma}) is satisfied, then implicit function theorem can be used in conjunction with (\ref{Gsigma}) to imply the existence of a unique continuously differentiable function  $\sigma(A;\beta,\theta_c=\pi)$ for $(A,\sigma)\in\Omega_{\tilde{\epsilon}}$ such that  $G(A,\sigma(A;\beta,\theta_c=\pi);\beta,\theta_c=\pi)=0$, which is bounded from above and below by the curves $\sigma=\sqrt{1-A^2}$ and $\sigma=0$. These results allow us to demonstrate that
\begin{equation} \label{sigma1-A}
\sigma = \frac{G_A}{G_{\sigma}}(1-A)+ O( (1-A)^2).
\end{equation}

  \bigskip

  From (\ref{implicitG}) we obtain
  that
  $$G_{\sigma}= [(1-k^2)F_{\phi}(\phi,k) - E_{\phi}(\phi,k)] \phi_{\sigma} + [-2k F(\phi,k) + (1-k^2)F_k(\phi,k) -E_k(\phi,k)]k_{\sigma}\\[1.5ex]$$
$$ -3[-2k K(k) + (1-k^2) K'(k) - E'(k)] k_{\sigma}  - [2 \lambda k \bar{z} + k \cos(\bar{\theta}_1) + k]_{\sigma},$$
 where $\phi = \bar{\theta}_1 +\frac{\pi}{2},$  and similarly that
$$G_{A}= [(1-k^2)F_{\phi}(\phi,k) - E_{\phi}(\phi,k)] \phi_{A} + [-2k F(\phi,k) + (1-k^2)F_k(\phi,k) -E_k(\phi,k)]k_{A}\\[1.5ex]$$
$$ -3[-2k K(k) + (1-k^2) K'(k) - E'(k)] k_{A}  - [2 \lambda k \bar{z} + k \cos(\bar{\theta}_1) + k]_{A}.$$

  From \cite[\S19.2, \S19.4]{Functions},
  \begin{equation} \label{completek}
  K'(k)=\frac{1}{k (k')^2}[E(k) - (k')^2 K(k)], \quad E'(k)=\frac{1}{k}[E(k)-K(k)],
  \end{equation}
  \begin{equation} \label{incompletephi}
  F_\phi(\phi,k)=\frac{1}{\sqrt{1-k^2 \sin^2 \phi}}, \quad E_\phi(\phi,k)=\sqrt{1-k^2 \sin^2 \phi},
  \end{equation}
  \begin{equation} \label{incompletek}
  F_k(\phi,k)=\frac{1}{k (k')^2}[E(\phi,k) - (k')^2 F(\phi,k)]-\frac{k \sin( \phi)\cos(\phi)}{(k')^2\sqrt{1-k^2 \sin^2(\phi)}}, \quad E_k(\phi,k)=\frac{1}{k}[E(\phi,k)-F(\phi,k)].
  \end{equation}

Using (\ref{completek})--(\ref{incompletek}) and since $\phi = \bar{\theta}_1 +\frac{\pi}{2}$, we obtain that
$$G_{\sigma}= - \frac{k^2 \sin^2(\bar{\theta}_1)}{\sqrt{1 - k^2 \cos^2(\bar{\theta}_1)}} (\bar{\theta}_1)_{\sigma} -k \Biggl[F(\bar{\theta}_1 +\frac{\pi}{2},k) - \frac{ \sin(\bar{\theta}_1) \cos(\bar{\theta}_1)}{\sqrt{1-k^2 \cos^2(\bar{\theta}_1)}}-3K(k)\Biggr]  k_{\sigma}
 - [2 \lambda k \bar{z} +  k \cos(\bar{\theta}_1) + k]_{\sigma},$$
    $$G_{A}= - \frac{k^2 \sin^2(\bar{\theta}_1)}{\sqrt{1 - k^2 \cos^2(\bar{\theta}_1)}} (\bar{\theta}_1)_{A} -k \Biggl[F(\bar{\theta}_1 +\frac{\pi}{2},k) - \frac{ \sin(\bar{\theta}_1) \cos(\bar{\theta}_1)}{\sqrt{1-k^2 \cos^2(\bar{\theta}_1)}}-3K(k)\Biggr]  k_{A}  - [2 \lambda k \bar{z} +  k \cos(\bar{\theta}_1) + k]_{A}.$$

\bigskip
We now wish to show that $k$ is small and to estimate its size. From (\ref{parameters}),(\ref{kk'}) we obtain that 
\begin{equation} \label{Psi}
\frac{1}{k^2}=\Psi(A,\sigma;\beta)=1 - 8 A \lambda\cos(\beta) + 4 \lambda^2.
\end{equation}
Hence
\begin{equation} \label{ksigma}
\Bigl(\frac{1}{k^2}\Bigr)_\sigma = \Psi_\lambda \lambda_{\sigma} = 8 (\lambda-A \cos(\beta)) \lambda_\sigma.
\end{equation}
Since $\tilde{\epsilon}$ is assumed to be sufficiently small so that condition (I) is satisfied within $\Omega_{\tilde{\epsilon}}$, we obtain
\begin{equation} \label{k2}
 \Psi_\lambda  = 8\left(\frac{(A^2+\sigma^2)A\cos(\beta)+\sigma\sin(\beta)}{1-A^2-\sigma^2}\right)<8A\cos(\beta)<0.
\end{equation}
Note further that
 \begin{equation} \label{k_sigma}
 \lambda_\sigma=\frac{\cos(\beta)( 2 \sigma A + \tan(\beta)(1 - A^2 -\sigma^2 + 2 \sigma^2))}{(1 -A^2 -\sigma^2)^2}.
 \end{equation}
 Thus, by evaluating $\lambda_\sigma$ in the limits $\sigma \downarrow 0,$ $\sigma \uparrow \sqrt{1-A^2}$ we obtain 
 $$\lim_{\sigma \downarrow 0} \lambda_\sigma=\frac{\sin(\beta)}{(1-A^2)} >0, \quad \lim_{\sigma \uparrow \sqrt{1-A^2}} \lambda_\sigma =\lim_{\sigma \uparrow \sqrt{1-A^2}}\Biggl[ \frac{2 \sigma (\sigma \sin(\beta)+A \cos(\beta))}{(1-A^2 -\sigma^2)^2} +\frac{\sin(\beta)}{(1-A^2 -\sigma^2)}\Biggr]<0,$$
which implies that $\lambda_\sigma$ changes sign as $\sigma$ varies for fixed $A$.

\smallskip
Note that it also follows from (\ref{Psi}) that
\begin{equation} \label{kA}
\Bigl(\frac{1}{k^2}\Bigr)_A = \Psi_\lambda \lambda_{A} - 8 \lambda\cos(\beta) = 8 (\lambda-A \cos(\beta)) \lambda_A  - 8\lambda \cos(\beta),
\end{equation}
where
\begin{equation}\label{XAl}
\lambda_A = \frac{\cos(\beta)((1+ A^2 -\sigma^2) + 2 \sigma A \tan(\beta))}{(1-A^2 -\sigma^2)^2}<\frac{ \cos(\beta)((1+ A^2 -\sigma^2 ) + 2  A^2 (1-2A^2 -2 \sigma^2))}{(1-A^2 -\sigma^2)^2}
= \frac{\cos(\beta)(1  + 4 A^2)}{1-A^2 -\sigma^2}<0.
\end{equation}
Combining (\ref{k2}),(\ref{kA}),(\ref{XAl}) yields
that
$$\Bigl(\frac{1}{k^2}\Bigr)_A> \frac{ 8A\cos^2(\beta)(1  + 4 A^2)}{1-A^2 -\sigma^2}-\frac{8(A \cos^2(\beta) +\sigma \cos(\beta)\sin(\beta))}{1-A^2 -\sigma^2}=
\frac{ 32  A^3 \cos^2(\beta)-8\sigma \cos(\beta)\sin(\beta)}{1-A^2 -\sigma^2}>0.$$
 So $0<k(A,\sigma;\beta)<k(A_0,\sigma;\beta),$ where $A_0:=1-\tilde{\epsilon}$.
 Thus, by (\ref{Psi}),(\ref{k2}) we obtain 
 \begin{equation} \label{kboundm}
  0<k(A,\sigma;\beta) < k(A_0,\sigma;\beta)\le k(A_0,\sigma_{\max};\beta), \quad \sigma_{\max}: =\argmax_{0 < \sigma < \sqrt{1-A_0^2}}\lambda(A_0,\sigma;\beta).
  \end{equation}
  Let us define 
  \begin{equation}\nonumber \bar{\lambda}_{\max}:=\frac{A_0 \cos(\beta) +\sqrt{1-A_0^2}\sin(\beta)}{1-A_0^2},
  \end{equation} 
  which, clearly, satisfies $\bar{\lambda}_{\max}  > \lambda(A_0,\sigma_{\max};\beta).$ Thus, by defining $\tilde{k}(A,\lambda(A,\sigma;\beta)):=k(A,\sigma;\beta),$ and following the evaluations given in \cite[\S1.3]{MScKG} we obtain that by requiring 
  \begin{equation} \label{Xbound}
  \tan(\beta)-1>-\frac{(1-8\tilde{\epsilon})}{2\sqrt{\tilde{\epsilon}}},\ \tilde{\epsilon} < 1+\frac{1}{2\cos(\beta)}\quad\text{if}\quad \beta\in\left(\frac{2\pi}{3},\pi\right),
  \end{equation}
  we can extend the domain of $\tilde{k}(A_0;\lambda)$ to all $\lambda\le \bar{\lambda}_{\max},$  and in this domain $\tilde{k}(A_0,\lambda)>0,$ $\tilde{k}_\lambda(A_0,\lambda)>0,$ thus yielding the following upper bound:
  \begin{equation} \label{kbound}
  0<k(A,\sigma;\beta)<\tilde{k}(A_0,\bar{\lambda}_{\max}) = \frac{\tilde{\epsilon}}{-\cos(\beta)}+O(\tilde{\epsilon}^{3/2}).
  \end{equation}

\bigskip\par\noindent
Thus, by making use of the expansions (\ref{incomplete2})--(\ref{complete2}) we obtain the following expansions for $G_\sigma,\,G_A:$
$$G_{\sigma}= \frac{- k^2 \sin^2(\bar{\theta}_1)}{\sqrt{1 - k^2 \cos^2(\bar{\theta}_1)}} (\bar{\theta}_1)_{\sigma} + \Bigl[ [-\bar{\theta}_1 + \pi]k +O(k^3) + \frac{k \sin(\bar{\theta}_1) \cos(\bar{\theta}_1)}{\sqrt{1-k^2 \cos^2(\bar{\theta}_1)}}\Bigr]k_{\sigma}
  - [2 \lambda k \bar{z} + k \cos(\bar{\theta}_1) + k]_{\sigma},$$
  and similarly that
$$G_{A}=  \frac{-k^2 \sin^2(\bar{\theta}_1)}{\sqrt{1 - k^2 \cos^2(\bar{\theta}_1)}} (\bar{\theta}_1)_{A} + \Bigl[ [-\bar{\theta}_1 + \pi]k +O(k^3) + \frac{k \sin(\bar{\theta}_1) \cos(\bar{\theta}_1)}{\sqrt{1-k^2 \cos^2(\bar{\theta}_1)}}\Bigr]  k_{A}
  - [2 \lambda k \bar{z} + k \cos(\bar{\theta}_1) + k]_{A}.$$
   Simplifying,
  $$G_{\sigma}=[ k \sin(\bar{\theta}_1) +O(k^2) ] (\bar{\theta}_1)_{\sigma} + [ -2\lambda\bar{z} -[\cos(\bar{\theta}_1)+1]+O(k)]k_{\sigma}
  - 2k [\lambda \bar{z}]_{\sigma},$$
  and similarly
  $$G_{A}=  [ k \sin(\bar{\theta}_1)  +O(k^2) ] (\bar{\theta}_1)_{A} + [-2\lambda\bar{z} -[\cos(\bar{\theta}_1)+1]+ O(k)]  k_{A}
  -2k [\lambda\bar{z}]_{A}.$$

\bigskip\par
To evaluate $G_A$ and $G_\sigma$ to leading order, we proceed to evaluate the various contributions.

\bigskip\par
Recalling that $\bar{\theta}_1=\bar{\theta}_3-\beta$, where $\bar{\theta}_3=\arcsin\Bigl(\frac{A}{\sqrt{\sigma^2 + A^2}}\Bigr)$, with $\bar{\theta}_3\in (0, \frac{\pi}{2}),$
we find that
$$1+O(\tilde{\epsilon}) = \frac{A}{\sqrt{\sigma^2 + A^2}}= \sin(\bar{\theta}_3)=\sin( \frac{\pi}{2} +( \bar{\theta}_3-\frac{\pi}{2}))= 1 -\frac{1}{2} ( \bar{\theta}_3-\frac{\pi}{2})^2 + O( ( \bar{\theta}_3-\frac{\pi}{2})^4),$$
and hence that
$$\bar{\theta}_3=\frac{\pi}{2} +O(\tilde{\epsilon}^{1/2}), \quad \bar{\theta}_1=\frac{\pi}{2}-\beta +O(\tilde{\epsilon}^{1/2}),$$
$$\sin(\bar{\theta}_1)=\frac{A  \cos(\beta)-  \sigma \sin(\beta)}{\sqrt{\sigma^2+A^2}} =\cos(\beta) +O(\tilde{\epsilon}^{1/2}),\quad
\cos(\bar{\theta}_1)=\frac{\sigma  \cos(\beta) + A \sin(\beta)}{\sqrt{\sigma^2+A^2}}=\sin(\beta) +O(\tilde{\epsilon}^{1/2}),$$
and that
$$(\bar{\theta}_1)_\sigma= \frac{\sqrt{\sigma^2 + A^2}}{\sigma} \cdot \frac{-\sigma A}{(\sigma^2 + A^2)^{3/2}}= \frac{-A}{\sigma^2 + A^2}=-1 +O(\tilde{\epsilon}),\quad
 (\bar{\theta}_1)_A= \frac{\sqrt{\sigma^2 + A^2}}{\sigma} \cdot \frac{\sigma^2}{(\sigma^2 + A^2)^{3/2}}= \frac{\sigma}{\sigma^2 + A^2}=O(\tilde{\epsilon}^{1/2}).$$

\bigskip\bigskip\par\noindent Using the above,
%
$$G_{\sigma}=-k \cos(\beta) +O({\tilde{\epsilon}}^{3/2})  + [ -2\lambda\bar{z} - (\cos(\bar{\theta}_1)+1)+O(k)]k_{\sigma}
  -2k [\lambda\bar{z}]_{\sigma},$$
$$G_{A}=  O(\tilde{\epsilon}^{3/2})+ [-2\lambda\bar{z}  - (\cos(\bar{\theta}_1)+1)+O(k)]  k_{A}
   -2k [\lambda\bar{z}]_{A}.$$
Note that we may also write the above as
$$G_{\sigma}=-k \cos(\beta) +O({\tilde{\epsilon}}^{3/2})  - 2\bar{z} [ \lambda k ]_{\sigma}  -[1+ \sin(\beta) + O({\tilde{\epsilon}}^{1/2})]k_{\sigma}   - 2k \lambda \bar{z}_{\sigma},$$
$$G_{A}=  O({\tilde{\epsilon}}^{3/2}) -  2\bar{z} [\lambda k]_{A}   -[1+ \sin(\beta) + O({\tilde{\epsilon}}^{1/2})]k_{A}   - 2k \lambda\bar{z}_{A}.$$

\bigskip\par Recalling that $\bar{z}=A\,\acosh({\sqrt{\sigma^2 + A^2}}/{A})$, and that $\acosh'(x)=\frac{1}{\sqrt{x^2-1}}$ for $x > 1$, we get that
$$\bar{z}_\sigma =\frac{A}{\sqrt{\sigma^2 + A^2}}, \quad \bar{z}_A=\acosh\Bigl( \frac{\sqrt{\sigma^2 + A^2}}{A}\Bigr) - \frac{\sigma}{\sqrt{\sigma^2 + A^2}}.$$
From the above
$$0<\bar{z} =A\,\acosh\Bigl(\frac{\sqrt{\sigma^2 + A^2}}{A}\Bigr)<  A\,\acosh\Bigl(\frac{1}{A}\Bigr)<  \acosh\Bigl(\frac{1}{A_0}\Bigr)=:y, \quad A_0=1-\tilde{\epsilon} < A <1.$$
And noting that as $0<\tilde{\epsilon} \ll 1,$  $A_0=1-\tilde{\epsilon},$ and $0<y \ll 1$, 
$$1 + \frac{1}{2}y^2+O(y^4) = \cosh(y)=(A_0)^{-1}=1 + \tilde{\epsilon} + O(\tilde{\epsilon}^2) \quad\Rightarrow \quad y= \sqrt{ 2 \tilde{\epsilon}} +  O(\tilde{\epsilon}^{3/2}),$$
and hence
$$\bar{z}= O(\tilde{\epsilon}^{1/2}).$$
Similarly $1-\tilde{\epsilon}=A_0 <A< \bar{z}_\sigma < 1$ implies that $$\quad \bar{z}_\sigma=1 + O(\tilde{\epsilon}).$$ 
Also it is easy to check that $\bar{z}_{A\, \sigma}>0,\ \left(\acosh(A^{-1})- \sqrt{1-A^2}\right)_A<0$ and hence
$$0 < \bar{z}_A < \acosh(A^{-1})- \sqrt{1-A^2} < \acosh(A_0^{-1})- \sqrt{1-A_0^2} \quad\Rightarrow\quad 0<\bar{z}_A < \sqrt{ 2 \tilde{\epsilon}}- \sqrt{ 2 \tilde{\epsilon}}+ O(\tilde{\epsilon}^{3/2})=O(\tilde{\epsilon}^{3/2}).$$

\bigskip\par\noindent
We now turn to estimating $\lambda k$ and its derivatives, to allow estimation of $G_A$, $G_\sigma$ above.
Note that in terms of the notation above,
\begin{equation} \label{Xk2}
(\lambda k)^{-2} = \lambda^{-2} \Psi(A, \sigma;\beta), \quad \Psi =1 - 8 A \lambda\cos(\beta) + 4 \lambda^2.
\end{equation}
Since from (\ref{parameters})
\begin{equation}\label{lambda-1small}
0<\left|\frac{1}{\lambda}\right|=\frac{1-A^2-\sigma^2}{-A\cos(\beta)-\sigma\sin(\beta)}=O(\tilde{\epsilon}),
\end{equation}
using (\ref{Xk2}) we obtain,
\begin{equation} \label{Xk0}
\lambda k=-\frac{1}{2}\Bigl[ 1 + O(\lambda^{-1})\Bigr]=-\frac{1}{2}[1+ O(\tilde{\epsilon})].
\end{equation}
Using the estimates for $\bar{z}_\sigma$, $\bar{z}_A$, as well as (\ref{Xk0}), in $G_\sigma$, $G_A$ yields
$$G_{\sigma}=-k \cos(\beta) +O({\tilde{\epsilon}}^{3/2})  - 2\bar{z} [ \lambda k ]_{\sigma}   -[1+ \sin(\beta) + O({\tilde{\epsilon}}^{1/2})]k_{\sigma}   + 1+O(\tilde{\epsilon}),$$
$$G_{A}=    O({\tilde{\epsilon}}^{3/2})  - 2\bar{z} [\lambda k]_{A}  -[1+ \sin(\beta) + O({\tilde{\epsilon}}^{1/2})]k_{A}    + O({\tilde{\epsilon}}^{3/2}).$$

Also from (\ref{Xk2}), we get that
\begin{equation} \label{lk2X}
((\lambda k)^{-2})_\lambda = - 2 \lambda^{-3} \Psi + \lambda^{-2} \Psi_\lambda = -2\lambda^{-2}(\lambda^{-1} -4 A \cos{\beta}).
\end{equation}
Using (\ref{lk2X})
\begin{equation} \label{lk2sigma}
((\lambda k)^{-2})_\sigma = -2\lambda^{-2}(\lambda^{-1} -4 A \cos{\beta}) \lambda_\sigma,
\end{equation}
\begin{equation} \label{lk2A}
((\lambda k)^{-2})_A = -2\lambda^{-2}(\lambda^{-1} -4 A \cos{\beta}) \lambda_A + \lambda^{-2}\Psi_A=  -2\lambda^{-2}(\lambda^{-1} -4 A \cos{\beta}) \lambda_A - 8 \cos(\beta) \lambda^{-1}.
\end{equation}

\par\noindent
Estimates for $k_\sigma$, $k_A$, $(\lambda k)_\sigma$,  $(\lambda k)_A$ can be obtained by
$$k_\sigma = - \frac{1}{2} k^3 (k^{-2})_\sigma, \quad k_A = -\frac{1}{2} k^3 (k^{-2})_A,\quad 0<k<1,$$
$$(\lambda k)_\sigma=-\frac{1}{2} (\lambda k)^3 ((\lambda k)^{-2})_\sigma,\quad (\lambda k)_A=-\frac{1}{2} (\lambda k)^3 ((\lambda k)^{-2})_A,\quad \lambda k<0.$$

Thus, using (\ref{parameters}), (\ref{ksigma})--(\ref{kA}), and (\ref{Xk0}) we obtain:
$$k_\sigma =\frac{-4(\cos(\beta)+O({\tilde{\epsilon}}^{1/2}))O({\tilde{\epsilon}}^{1/2})}{(4\cos^2(\beta)+O({\tilde{\epsilon}}^{1/2}))^{3/2}}=O({\tilde{\epsilon}}^{1/2}),$$
$$k_A = \frac{-8\cos^2(\beta)+O({\tilde{\epsilon}}^{1/2})}{(4\cos^2(\beta)+O({\tilde{\epsilon}}^{1/2}))^{3/2}}+O({\tilde{\epsilon}}^{2})=\frac{1}{\cos(\beta)}(1+ O({\tilde{\epsilon}}^{1/2})).$$ 
And from (\ref{lambda-1small})--(\ref{lk2A})
$$(\lambda k)_\sigma = \left(\frac{1}{16}+O(\tilde{\epsilon})\right)\left(8A\,\cos(\beta)-\frac{2}{\lambda}\right)\left(\frac{\sin(\beta)(1 - A^2+\sigma^2) +2 \sigma A \cos(\beta)}{(A\,\cos(\beta)+\sigma\,\sin(\beta))^2}\right)=\left(\frac{1}{2}\cos(\beta)+O(\tilde{\epsilon})\right)O({\tilde{\epsilon}}^{1/2})=O({\tilde{\epsilon}}^{1/2}),$$
$$(\lambda k)_A = \left(\frac{1}{16}+O(\tilde{\epsilon})\right)\left(\left(8A\,\cos(\beta)-\frac{2}{\lambda}\right)\left(\frac{\cos(\beta)(1 + A^2-\sigma^2) +2 \sigma A \sin(\beta)}{(A\,\cos(\beta)+\sigma\,\sin(\beta))^2}\right)-8\cos(\beta)\frac{1}{\lambda}\right)=$$$$\left(\frac{1}{16}+O(\tilde{\epsilon})\right)\left(\frac{16\cos^2(\beta)}{\cos^2(\beta)}+O(\tilde{\epsilon}^{1/2})\right)=1+ O({\tilde{\epsilon}}^{1/2}).$$
\color{black} from which we may conclude
$$G_\sigma = 1 +  O({\tilde{\epsilon}}^{1/2}),\quad  G_A =\frac{\sin(\beta)+1}{-\cos(\beta)} + O({\tilde{\epsilon}}^{1/2}),$$
which implies that there is a unique solution $\sigma=\sigma(A;\beta,\theta_c=\pi)$ for each $1-\tilde{\epsilon} < A<1$. Since the upper and lower bounds noted earlier imply that $\lim_{A \rightarrow 1^-} \sigma(A; \beta, \theta_c=\pi)=0$, by the implicit function theorem it follows that 
$$
\sigma=\Bigl[\frac{\sin(\beta)+1}{-\cos(\beta)} + O({\tilde{\epsilon}}^{1/2}) \Bigr](1-A).$$ 

	Given the result above, one may conclude that $\sigma=O(1-A)$ within $\Omega_{\tilde{\epsilon}}$. Then, using (\ref{parameters}) and by reevaluating the various contributions in $G_\sigma,\,G_A$, one may obtain the more accurate estimates $$G_\sigma=1+O(1-A),\quad\quad G_A=\frac{\sin(\beta)+1}{-\cos(\beta)}+O(1-A),$$ which imply (\ref{SolutionCurve}).
\end{proof}
\begin{claim}\label{claim3}
If  $\tilde{\epsilon}>0$ is sufficiently small, then the locus of points described by (\ref{SolutionCurve}) satisfies condition $\rm{(IV)}$ within $\Omega_{\tilde{\epsilon}}$.
\end{claim}
\begin{proof}[Proof of Claim \ref{claim3}]
Note that since $\tan(\beta)<0,$ the following would be a sufficient requirement 
$$1-A^2\ge 2\sigma^2=2\Bigl[\frac{\sin(\beta)+1}{\cos(\beta)}\Bigr]^2(1-A)^2+O((1-A)^3).$$ Since for each $\beta\in\left(\frac{\pi}{2},\pi\right)$ we can find small enough $0<\tilde{\epsilon}\ll 1$ such that $0<\sigma(A;\beta,\theta_c=\pi)<\frac{2}{-\cos(\beta)}(1-A)$ for all $A\in(1-\tilde{\epsilon},1)$, and such that $\tilde{\epsilon}<\frac{2\cos^2(\beta)}{\cos^2(\beta)+8}$, we obtain
$$1+A>2-\tilde{\epsilon}>\frac{8}{\cos^2(\beta)}\tilde{\epsilon}> \frac{2\sigma^2}{1-A}.$$ Thus condition (IV) is satisfied by the locus of points given in (\ref{SolutionCurve}) within  $\Omega_{\tilde{\epsilon}}$.
\end{proof}
\end{proof}
\color{black}

\section{Conclusions} While our study  of the steady states is so far not yet exhaustive, we have established a parametric framework which allows for their complete characterization.   Clearly many dynamic questions  lie ahead which are of  concern to materials scientists, such as the role of the size and shape of grains which border on or lie in close proximity to holes \cite{Thompson2012}. In particular, many analytic stability questions can be readily formulated \cite{Kohsaka,Wetton}, such as the stability
of the steady states with respect to non-axi-symmetric perturbations.

\bigskip\par\noindent\textbf{Acknowledgments:} The authors acknowledge the support of the Israel Science Foundation (Grant \#1200/16).

\section{Appendix}
\subsection{Legendre's Integrals}\label{Legendre}

In \S\ref{asymptotic analysis} we made use of various  results regarding \textbf{Legendre elliptic} integrals which are given below. These results can be found, for example, in \cite[\S19.1, \S19.2]{Functions}.

\smallskip
Given an argument $\phi \in \mathbb{C}$ and given a modulus $k\in\mathbb{C}$ such that $1-\sin^2(\phi)\in \mathbb{C}\backslash (-\infty,0]$ and \\$1-k^2\sin^2(\phi)\in \mathbb{C}\backslash (-\infty,0]$, except that one of them may be $0$, then
\begin{equation}
F(\phi,k):=\int_{0}^{\sin(\phi)}{\frac{dt}{\sqrt{1-t^2}\sqrt{1-k^2t^2}}},\quad
E(\phi,k):=\int_{0}^{\sin(\phi)}{\frac{\sqrt{1-k^2t^2}}{\sqrt{1-t^2}}dt,}\label{incomplete}
\end{equation}
correspond to Legendre's incomplete integrals of the first and second kinds, respectively.
Setting $\phi=\frac{\pi}{2}$ in (\ref{incomplete}) yields Legendre's complete integrals of the first and second kinds, respectively,
\begin{equation}
K(k):=F(\frac{\pi}{2},k),\quad
E(k):=E(\frac{\pi}{2},k).\label{complete}
\end{equation}
The more general definition can be seen in \cite[\S19.2]{Functions}.
\smallskip

\smallskip

\smallskip
\color{black} Asymptotic evaluations of the incomplete and complete Legendre elliptic integrals are well known from the literature.
\smallskip
Using \cite[\S2.1, \S2.2, \S2.3]{Expansion}, we can write the asymptotic expansions of $F(\phi,k),\ E(\phi,k)$ for $0<k\ll 1,\ -\frac{\pi}{2}<\phi<\frac{\pi}{2}$ as follows, 
\begin{align}\label{incomplete2}
\begin{array}{l}
F(\phi,k)=\phi+\frac{\phi-\sin(\phi)\cos(\phi)}{4}k^2+O(k^4)\\\\
E(\phi,k)=\phi-\frac{\phi-\sin(\phi)\cos(\phi)}{4}k^2+O(k^4)\\\\
F(\phi,k)-E(\phi,k)=\frac{\phi-\sin(\phi)\cos(\phi)}{2}k^2+O(k^4).
\end{array}
\end{align}
\smallskip
In addition, using \cite[\S19.5.1, \S19.5.2]{Functions}, we obtain 
asymptotic expansions for $0<k\ll 1,$ 
\begin{align}\label{complete2}
\begin{array}{l}
K(k)=\frac{\pi}{2}+\frac{\pi}{8}k^2+O(k^4)\\\\
E(k)=\frac{\pi}{2}-\frac{\pi}{8}k^2+O(k^4)\\\\
K(k)-E(k)=\frac{\pi}{4} k^2+O(k^4).
\end{array}
\end{align}

\bigskip
\subsection{Expressing $z_1(\theta_1)$ in terms of Legendre's Integrals}\label{LegendreZ1}
Now, we shall express $z_1(\theta_1)$ for $\theta_1\in(0,\pi]$ in terms of the Legendre's integrals defined above. Let
\begin{equation}
k:=\frac{1}{\sqrt{1+4\lambda^2 a_\ell^2}},\quad k':=\sqrt{1-k^2}.\label{kk'}
\end{equation}
From (\ref{lambda})--(\ref{al}) and (\ref{kk'})
\begin{equation}
k=\frac{1-A^2-\sigma^2}{\sqrt{(1-A^2-\sigma^2)^2+4(A\,\cos(\beta)+\sigma\sin(\beta))((2A^2+2\sigma^2-1)A\,\cos(\beta)+\sigma\sin(\beta))}},\quad 0<k<1,\label{k}
\end{equation}
\begin{equation}
k'=\frac{\sqrt{4(A\,\cos(\beta)+
		\sigma\sin(\beta))((2A^2+2\sigma^2-1)A\,\cos(\beta)+\sigma\sin(\beta))}}{\sqrt{(1-A^2-\sigma^2)^2
		+4(A\,\cos(\beta)+\sigma\sin(\beta))((2A^2+2\sigma^2-1)A\,\cos(\beta)+\sigma\sin(\beta))}},\quad 0<k'<1.\label{kk}
\end{equation}

By (\ref{solutions}) we obtain for $\theta_1\in(0,\pi]$
\begin{eqnarray}\label{z1thetac}
&z_1(\theta_1)= \bar{z} - \frac{1}{2\lambda} \int_{\bar{\theta}_1}^{\theta_1}{\Biggl[\frac{\sin^2(x)}{\sqrt{\sin^2(x)+4\lambda^2a_\ell^2}}-\sin(x)\Biggr]dx}\\
& =- \frac{1}{2\lambda}\left( \int_{\bar{\theta}_1}^{0}{\Biggl[\frac{\sin^2(x)}{\sqrt{\sin^2(x)+4\lambda^2a_\ell^2}}\Biggr]dx}+\int_{0}^{\theta_1}{\Biggl[\frac{\sin^2(x)}{\sqrt{\sin^2(x)+4\lambda^2a_\ell^2}}\Biggr]dx}+\cos(\theta_1)-\cos(\bar{\theta}_1)-2\lambda \bar{z}\right)\nonumber.
\end{eqnarray}
Noting the definitions of $k$ and $k'$ given in (\ref{kk'}), the definitions of  $F(\phi,k)$, $E(\phi,k)$  given in (\ref{incomplete}), and the definitions of $K(k)$, $E(k)$  given in (\ref{complete}), and recalling that $\bar{\theta}_1\in(-\pi,0)\subset[-\pi,0],$
$$\int_{\bar{\theta}_1}^{0}{\frac{\sin^2(x)}{\sqrt{\sin^2(x)+4\lambda^2a_\ell^2}}\,dx}
\underset{\cos(x)\rightarrow{t}}{=}\int_{\cos(\bar{\theta}_1)}^{1}{\frac{1-t^2}{\sqrt{(1-t^2)^2+(1-t^2)(4\lambda^2a_\ell^2)}}\,dt}=
k\int_{\cos(\bar{\theta}_1)}^{1}{\frac{1-t^2}{\sqrt{(1-t^2)(1-k^2t^2)}}\,dt}$$
$$\quad =-\frac{k'^2}{k}\int_{\cos(\bar{\theta}_1)}^{1}{\frac{dt}{\sqrt{(1-t^2)(1-k^2t^2)}}}+
\frac{1}{k}\int_{\cos(\bar{\theta}_1)}^{1}{\frac{\sqrt{1-k^2t^2}}{\sqrt{1-t^2}}dt}
=\frac{k'^2}{k}\left(F\left(\bar{\theta}_1+\frac{\pi}{2},k\right)-K(k)\right)-\frac{1}{k}\left(E\left(\bar{\theta}_1+\frac{\pi}{2},k\right)-E(k)\right).
$$
Similarly, since $-\theta_1\in[-\pi,0)\subset[-\pi,0],$
$$\int_{0}^{\theta_1}{\frac{\sin^2(x)}{\sqrt{\sin^2(x)+4\lambda^2a_\ell^2}}\,dx} \underset{x\rightarrow{-y}}{=}\int_{-\theta_1}^{0}{\frac{\sin^2(y)}{\sqrt{\sin^2(y)+4\lambda^2a_\ell^2}}\,dy}=\frac{k'^2}{k}\left(F\left(\frac{\pi}{2}-\theta_1,k\right)-K(k)\right)-\frac{1}{k}\left(E\left(\frac{\pi}{2}-\theta_1,k\right)-E(k)\right).$$
Hence (\ref{z1thetac}) implies that for $\theta_1\in(0,\pi]$
\begin{align}
\begin{array}{l}
z_1(\theta_1)=\frac{-1}{2\lambda}\Bigg(\frac{k'^2}{k}\Bigg(F\left(\frac{\pi}{2}-\theta_1,k\right)+F\left(\bar{\theta}_1+\frac{\pi}{2},k\right)-2K(k)\Bigg)\\-\frac{1}{k}\Bigg(E\left(\frac{\pi}{2}-\theta_1,k\right)+E\left(\bar{\theta}_1+\frac{\pi}{2},k\right)-2E(k)\Bigg)+\cos(\theta_1)-\cos(\bar{\theta}_1)-2\lambda \bar{z}\Bigg).\label{GG}
\end{array}
\end{align}
And in particular for $\theta_1=\pi$
\begin{align}
\begin{array}{l}
z_1(\pi)=\frac{-1}{2\lambda}\left(\frac{k'^2}{k}\Bigg(F\left(\bar{\theta}_1+\frac{\pi}{2},k\right)-3K(k)\Bigg)-\frac{1}{k}\Bigg(E\left(\bar{\theta}_1+\frac{\pi}{2},k\right)-3E(k)\Bigg)-1-\cos(\bar{\theta}_1)-2\lambda \bar{z}\right).\label{Gpi}
\end{array}
\end{align}\color{black}


\bibliographystyle{abbrv}
\bibliography{KG_ANC_YV_FINAL}

\end{document}